\documentclass[11pt,leqno]{article}

\usepackage{amsfonts,amsmath,amsthm}
\usepackage{graphicx,color,a4wide}
\usepackage{enumerate}

\def\R{\mathbb{R}}
\def\Z{\mathbb{Z}}
\def\N{\mathbb{N}}

\def\eps{\varepsilon}
\def\e{\varepsilon}
\def\a{\alpha}
\def\I{I}

\newtheorem{theo}{Theorem}[section]
\newtheorem{lem}[theo]{Lemma}
\newtheorem{pro}[theo]{Proposition}

\theoremstyle{definition}

\newtheorem{rem}[theo]{Remark}

\numberwithin{equation}{section}

\begin{document}

\title{\bf Steady state and long time convergence of spirals
moving by forced mean curvature motion}

 \author{N. Forcadel\footnote{INSA de Rouen, Normandie Universit\'e,
     Labo. de Math\'ematiques de l'INSA - LMI (EA 3226 - FR CNRS 3335)
     685 Avenue de l'Universit\'e, 76801 St Etienne du Rouvray cedex.
     France}, C. Imbert\footnote{CNRS, UMR 7580, Universit\'e
     Paris-Est Cr\'eteil, 61 avenue du G\'en\'eral de Gaulle, 94 010
     Cr\'eteil cedex, France}, R. Monneau
 \footnote{Universit\'e Paris-Est, CERMICS (ENPC), 
6-8 Avenue Blaise Pascal, Cit\'e Descartes, Champs-sur-Marne,
F-77455 Marne-la-Vall\'ee Cedex 2, France}
}

\maketitle


\begin{abstract}
In this paper, we prove the existence and uniqueness of a ``steady''
spiral moving with forced mean curvature motion.  This spiral has a
stationary shape and rotates with constant angular velocity.  Under
appropriate conditions on the initial data, we also show the long time
convergence (up to some subsequence in time) of the solution of the
Cauchy problem to the steady state.  This result is based on a
new Liouville result which is of independent interest.
\end{abstract}

\paragraph{AMS Classification:} 35K55, 35K65, 35A05, 35D40.

\paragraph{Keywords:} spirals, steady state, mean curvature
motion, Liouville theorem, long time convergence, motion of
interfaces, viscosity solutions.


\section{Introduction}

In this paper we are interested in curves in $\R^2$ which are
half-lines attached at the origin.  These lines are assumed to move
with normal velocity
\begin{equation}\label{eq:law}
V_n= 1 + \kappa
\end{equation}
where $\kappa$ is the curvature of the line. We assume that these
curves $\Gamma_t$ can be parametrized in polar coordinates as follows
\[
\Gamma_t = \{(r\cos\theta, r\sin\theta),\quad 
\mbox{such that}\quad r\ge 0,\quad \theta =  -U (t,r)\}.
\]
On the one hand, the Geometric Law~\eqref{eq:law} holds true if
$U$ satisfies 
\[ U_t = (1 + \kappa_U) |\nabla U|.\]
On the other hand, it is known (see for instance
\cite{spirale1}) that the curvature of the parametrized curve
$\Gamma_t$ has the following form
\begin{equation}\label{def:courbure}
 \kappa_{U}(t,r)=U_r\left(\frac {2+(r
  U_r)^2}{(1+(rU_r)^2)^{\frac32}}\right) + \frac{r
  U_{rr}}{(1+(rU_r)^2)^{\frac 32}}.
\end{equation}
Hence, the function $U$ has to satisfy the following quasi-linear parabolic
equation in non-divergence form for $(t,r)\in (0,+\infty)\times
(0,+\infty)$:
\begin{equation}\label{eq:main}
rU_t=\sqrt{1 + r^2U_r^2} + U_r \left(\frac{2+ r^2 U_r^2}{1+ r^2
  U_r^2}\right) + \frac{rU_{rr}}{1 + r^2U_r^2}
\end{equation}
supplemented with the following initial condition for $r\in (0,+\infty)$
\begin{equation}\label{eq:ci}
U(0,r)=U_0 (r).
\end{equation}

\subsection{Main results}

In \cite{spirale1}, we were able to prove an existence and uniqueness
result for equation \eqref{eq:main}-\eqref{eq:ci}.  We improve it by
proving in particular that solutions are regular up to the boundary
$r=0$.
\begin{theo}[Existence and uniqueness for the Cauchy problem]\label{th:cauchy}
Assume that $U_0\in W^{2,\infty}_{loc}(0,+\infty)$ is globally
Lipschitz continuous and satisfies
\[(U_0)_r\in W^{1,\infty}(0,+\infty) \quad \mbox{or}\quad 
\kappa_{U_0}\in L^\infty(0,+\infty)\]
and that there exists a radius $r_0>0$ such that
\[|1+ \kappa_{U_0}|\le Cr \quad \mbox{for}\quad 0\le r\le r_0.\]
Then there exists a globally Lipschitz continuous (in space and time) solution
$U$ such that
\[ U \in C^{1+\frac 16,2+\frac
  13}_{t,r}((0,+\infty)\times [0,+\infty)) \cap
  C^\infty((0,+\infty)\times (0,+\infty)).\] 
Moreover, for every $\delta>0$, $R>0$, there exists a constant
$C=C(\delta,R)$ such that for every $T\ge \delta>0$,
\[\|U-U(T,0)\|_{C^{1+\frac 16,2+\frac 13}_{t,r}([T, T+\delta]\times
  [0,R])}\le C.\]
Such a solution is unique in the class of continuous viscosity
solutions of \eqref{eq:main}-\eqref{eq:ci}. 
\end{theo}
\begin{rem}
In view of \eqref{def:courbure} and \eqref{eq:main}, the regularity of
$U$ stated in the previous theorem implies in particular that
\begin{equation}\label{eq::ob23}
\kappa_U + 1=0 \quad \mbox{at}\quad  r=0
\end{equation}
holds  for $t>0$. 
\end{rem}
\begin{rem}
The assumption that $U_0$ is globally Lipschitz was missing in the
statement of Theorem~1.7 in \cite{spirale1}.  We will recall below
(see Theorem~\ref{th::ob22}) the corrected version of this result.
\end{rem}
Our second main result is about the existence of a spiral with
stationary shape and rotating at constant speed. 
\begin{theo}[A steady state]\label{th:steady}
There exists a constant $\lambda\in\R$ and a globally Lipschitz
function continuous $\Phi$ in $[0,+\infty)$, satisfying
\[\lambda\ge 0 \quad \mbox{and}\quad \Phi_r \le 0 
\quad \mbox{on}\quad [0,+\infty)\]
such that $U(t,r)=\lambda t +\Phi(r)$ is a solution of (\ref{eq:main})
in $\R\times (0,+\infty)$. Moreover such a $\lambda$ is unique and
such a function $\Phi$ is unique up to addition of a constant. Moreover, the following properties hold true:
\begin{enumerate}[\upshape i)]
\item we have 
\[\frac14 \le \lambda\le \frac12.\]
\item $\Phi\in {C}^\infty(0,+\infty)\cap C^{2+\frac13}([0,+\infty))$
  satisfies for all $r \in [0,+\infty)$
\begin{align}
\label{eq:estGradR}
\displaystyle -\frac12\le \Phi_r&\le -\lambda, \\
\label{eq:profil courbure}
0 \le 1+\kappa_\Phi &\le \lambda r.
\end{align}
 Moreover $\Phi_r$ and the curvature $\kappa_\Phi$ are
  non-decreasing and 
\[\Phi_r(0)=-\frac12, \quad \Phi_r(+\infty)=-\lambda, \qquad
\kappa_\Phi(0)=-1, \quad \kappa_{\Phi}(+\infty)=0.\]
\item There exist some constants $a\in\R$ and $C>0$, such that
$\Phi$ satisfies  for all $r\in [0,+\infty)$
\[|\Phi(r)+\lambda r + \lambda \ln (1+r) -a | \le \frac{C}{1+r}.\]
\end{enumerate}
\end{theo}
\begin{rem}
Notice that the value of the angular velocity $\lambda$ have been
estimated to be $0.315$ by approximation in \cite{BCF}, and computed
to be $0.330958961$ by a shooting method in \cite{OR}.
\end{rem}
Our third  main result is concerned with the large time
behaviour of solutions of the Cauchy problem for initial data that are
``reasonably close'' to the steady state. 
\begin{theo}[Long time convergence]\label{thm:temps-long}
Under assumptions of Theorem~\ref{th:cauchy}, if the initial data $U_0$ 
further satisfies
\begin{equation}\label{eq::ob26}
|U_0 -\Phi| \le C
\end{equation}
and 
\begin{equation}\label{eq::ob27}
(U_0)_r \le \Phi_r\le -\lambda <0,
\end{equation}
where $(\lambda,\Phi)$ are given in Theorem~\ref{th:steady}, then for
any sequence $t_n \to +\infty$, there exists a subsequence (still
denoted by $t_n$) and a constant $a\in\R$ such that
\[U(t+t_n,r)-(\lambda(t+t_n)+\Phi(r)) \to a \quad 
\mbox{locally uniformly in $\R\times [0,+\infty)$}.\]
\end{theo}
\begin{rem} 
The fact that convergence only happens along a subsequence of times
 is expected.  Indeed a similar fact happens already for the
linear heat equation on the real line.  It is possible to cook up an
initial data which stays between $0$ and $1$ such that the solution
does not converges as times goes to infinity, but such that
convergence to a constant (locally uniformly) still happens for
subsequences in time (see in particular \cite[Lemma 8.6]{CE}). This happens here because we are working on the whole plane. 
On the contrary, when we work on the (compact) annulus (like in \cite{gik02}), there is a full convergence in time without taking a subsequence in time. 
\end{rem}
The proof of Theorem~\ref{thm:temps-long} is based on the following
Liouville result of independent interest.
\begin{theo}[Liouville result]\label{theo:liouville}
Let $U(t,r)$ be a globally Lipschitz continuous function (in space and
time) in $\R\times [0,+\infty)$.  We assume that $U$ is a global
  solution of \eqref{eq:main} in $\R \times (0,+\infty)$ and that
  there exists a constant $C>0$ such that the following holds:
\begin{equation}\label{eq:dist-finie}
|U(t,r) - \lambda t - \Phi(r) | \le C \quad \mbox{on}\quad \R\times [0,+\infty)
\end{equation}
where $(\lambda,\Phi)$ is given by Theorem \ref{th:steady}.  We also
assume that there exists some $\delta>0$ such that
\begin{equation}\label{eq::ob43}
U_r\le -\delta<0 \quad \mbox{in}\quad \R\times [0,+\infty).
\end{equation}
Then  
\[U(t,r) = \lambda t + \Phi(r) +a\]
for some constant $a \in \R$.
\end{theo}

\subsection{Review of the literature}

Spirals appear in several applications. Our main motivation comes from
continuum mechanics. In a two dimensional space, the seminal paper of
Burton, Cabrera and Frank \cite{BCF} studies the growth of crystals
with vapor.  When a screw dislocation line reaches the boundary of the
material, atoms are adsorbed on the surface in such a way that a
spiral is generated; moreover, under appropriate physical assumptions,
these authors prove that the geometric law governing the dynamics of
the growth of the spiral is precisely given by \eqref{eq:law}.  We
mention that there is an extensive literature in physics dealing with
crystal growth in spiral patterns (see for instance
\cite{sk99,RDCZWW}).  We also want to point out that motion of spirals
appear in other applications like in the modeling of the
Belousov-Zhabotinsky reagent \cite{murray}. To model the appearence of
such shapes, the reagent is modeled in \cite{Kee} by a system of
semi-linear parabolic equations; so-called spiral wave fronts
satisfying the geometric law \eqref{eq:law} can be formally derived.
The interested reader is also referred to e.g.
\cite{meron1,meron2,IIY}. \medskip
  
There exist different mathematical approaches to describe the motion
of spirals.  As far as we know, it appeared first in geometry in
\cite{A}. It was also used in order to study singularity formation
\cite{angenent,av}.  Other approaches have been used; for instance, a
phase-field approach was proposed in \cite{kp} and the reader is also
referred to \cite{FGT, on00,on03}. In \cite{gik02}, spirals moving in
(compact) annuli with homogeneous Neumann boundary condition are
constructed. From a technical point of view, the classical parabolic
theory is used to construct smooth solutions of the associated partial
differential equation; in particular, gradient estimates are
derived. We point out that in \cite{gik02}, the geometric law is
anisotropic, and is thus more general than \eqref{eq:law}. In
\cite{smereka00,ohtsuka03, ohtsukaPreprint,GNO}, the geometric flow is
studied by using the level-set approach. As in \cite{gik02}, the
authors of \cite{ohtsuka03,ohtsukaPreprint} consider spirals that
typically move inside a (compact) annulus and reaches the boundary
perpendicularly.  \medskip
  
Concerning the existence of ``steady'' spirals (in the case
where the exterior stress is zero), we refer to \cite{ishimura98}
where the construction is done by studying an ordinary differential
equation and to \cite{CGL} where the authors consider a two-point free
boundary problem for the curvature flow equation. We also refer to
\cite{gik02} where they construct a steady state on an
annulus using classical parabolic theory. In \cite{OR}, a numerical
computation of the angular velocity $\lambda$ of the spirals is
done. The authors find that the angular velocity is approximatively
$0,330958961$ (recall that we find that $\frac 14\le \lambda\le \frac
12$).

\subsection{Organization of the article}

In Section \ref{s2}, we prove that the solution has a certain
smoothness up to the boundary $r=0$, namely Theorem
\ref{th:cauchy}. In Section \ref{s3}, we construct the steady state,
first on an annulus and then on the whole space.  In Section \ref{s4},
we prove some asymptotics of any profile, and then deduce the
uniqueness of the profile (and of its angular velocity $\lambda$) as a
consequence of the asymptotics.  In Section \ref{s5}, we provide some
additional qualitative properties of the profile solution, including
monotonicity of its gradient and of its curvature. We also give a
bound from below on $\lambda$.  In Section \ref{s6}, we prove
Liouville theorem \ref{theo:liouville}.  In Section \ref{s7}, we prove
the long time convergence of the solution to the steady state
(up to addition of a constant), namely Theorem \ref{thm:temps-long}.
This result follows from Liouville Theorem and a gradient bound on the
solution (Proposition \ref{pro:estimate-gradient}) that is proven in
Section \ref{s7}.  Finally, Section \ref{s8} is an appendix where we
recall standard materials, like strong maximum principle, Hopf lemma,
Interior Schauder estimates.  We also prove a technical lemma
(Lemma~\ref{lem::ob35}) which is used in Section~\ref{s2}, and also
prove a result of independent interest which is not used in the rest
of the paper: the equation satisfied by the curvature of the graph of
the solution of the evolution problem.

\paragraph{Notation.} For a real number $a \in \R$, $a^+$ denotes $\max(a,0)$
and $a^-$ denotes $\max (-a,0)$. The ball of radius $r$ centered at
$x$ are denoted $B(x,r)$. If $x=0$, we simply write $B_r$.

\section{Regular solutions up to the origin}
\label{s2}

This section is devoted to the proof of Theorem~\ref{th:cauchy}. This
theorem improves \cite[Theorem 1.7]{spirale1} by establishing
regularity of solutions up to the origin. As we pointed out
previously, the assumption that $U_0$ is globally Lipschitz was
missing in the statement of \cite[Theorem 1.7]{spirale1}. This is
the reason why we first state a corrected version of this theorem.
\begin{theo}[Existence and uniqueness of smooth solutions for $r>0$,
    \cite{spirale1}]\label{th::ob22} 
Assume that $U_0\in W^{2,\infty}_{loc}(0,+\infty)$ is globally
Lipschitz continuous and satisfies
\[(U_0)_r\in W^{1,\infty}(0,+\infty) \quad \mbox{or}\quad \kappa_{U_0}\in L^\infty(0,+\infty)\]
and that there exists a radius $r_0>0$ such that
\[|1+ \kappa_{U_0}|\le Cr \quad \mbox{for}\quad 0\le r\le r_0.\]
Then there exists a unique viscosity solution $U$ of
\eqref{eq:main},\eqref{eq:ci} which is globally Lipschitz in space and
time.  Moreover this solution $U$ belongs to
$C^\infty((0,+\infty)\times (0,+\infty))$.
\end{theo}
In view of this result, proving Theorem~\ref{th:cauchy} amounts to
prove the following proposition. 
\begin{pro}[Space-time Lipschitz implies uniform regularity up to $r=0$]\label{pro::ob30}
Assume that $U$ is a globally Lipschitz continuous (in space and time) solution
of \eqref{eq:main} in $(0,+\infty)\times (0,+\infty)$.  Then $U(t,r)$
belongs to $C^{1+\frac 16,2+\frac 13}_{t,r}((0,+\infty)\times
[0,+\infty))$.  Moreover, for every $\delta>0$, $R>0$, there exists a
  constant $C=C(\delta,R)$ such that we have the following uniform
  bound for every $T\ge \delta>0$:
\begin{equation}\label{eq::ob40}
\|U-U(T,0)\|_{C^{1+\frac 16,2+\frac 13}_{t,r}([T, T+\delta]\times [0,R])}\le C.
\end{equation}
\end{pro}
Before proving this proposition, we get some useful a priori estimates
on the solution. 
\begin{lem}[A priori estimates]\label{lem:1}
Assume that $U$ is a globally Lipschitz continuous (in space and time) solution
of \eqref{eq:main} in $(0,+\infty)\times (0,+\infty)$, with Lipschitz
constant $L>0$.  Then $U \in C^\infty((0,+\infty) \times (0,+\infty))$
and there exists a constant $C=C(L)>0$ such that
for every $(t,r)\in (0,\infty)\times (0,\infty)$, we have
\begin{equation}\label{eq::ob31}
|U_r(t,r)+\frac 12|\le Cr\quad {\rm and}\quad |U_{rr}(t,r)|\le C(1+r^2).
\end{equation}
\end{lem}
\begin{proof}
We recall that we already proved in \cite[Theorem 1.7]{spirale1}
that $U\in C^\infty((0,+\infty)\times (0,+\infty))$.
We also recall that $U_t$ and $U_r$ are bounded, and that $U$ solves
\[rU_t=(1+\kappa_U)\sqrt{1+r^2U_r^2} .\]
We deduce that
\begin{equation}\label{est:1}
|1+\kappa_U|\le Cr
\end{equation}
for some constant $C$.  Remarking that
\[1+2U_r+rU_{rr}=(1+\kappa_U)(1+r^2U_r^2)^{\frac 32} -
r^2U_r^3-\left((1+r^2U_r^2)^{\frac 32}-1\right),\]
and using the  bound on $U_r$ and \eqref{est:1}, we deduce  that
\begin{eqnarray}
\nonumber |1+2 U_r + rU_{rr}| & \le & C(r+r^2+r^3 +r^4)\\
\label{est:2} & \le & C(r+r^4).
\end{eqnarray}
For fixed $t>0$, we set $\psi (r)= U(t,r) + r/2$ which satisfies $(r^2
\psi_r )_r = r(1+2 U_r + rU_{rr})$, and deduce that
\[ |(r^2 \psi_r )_r| \le C (r^2+r^5).\]
This implies $|r^2 \psi_r|\le C(r^3+r^6)$ and we finally get 
\begin{equation}\label{eq::ob62}
|U_r +\frac12 |=|\psi_r | \le C(r+r^4).
\end{equation}
Injecting this estimate in \eqref{est:2}, we finally get for all
$r\in(0,+\infty)$, $t\in (0,+\infty)$
\begin{equation}\label{eq::ob63}
|U_{rr}(t,r)|\le C(1+r^3).
\end{equation}
Because $U_r$ and $U_t$ are bounded, we can use (\ref{eq:main}) to get
for large $r$ that $|U_{rr}|\le Cr^2$. We can then improve
(\ref{eq::ob62}) and (\ref{eq::ob63}) to get (\ref{eq::ob31}).  This
ends the proof of the lemma.
\end{proof}

\begin{proof}[Proof of Proposition \ref{pro::ob30}]
The idea of the proof is to see $U$ as a radial solution of a partial
differential equation in three dimensions and to use the interior
regularity theory in 3D in order to deduce the boundary regularity up
to $r=0$.

More precisely, we set
\[V(t,X):=U(t,|X|) + \frac {|X|}2\quad {\rm for} \; X\in \R^3,\]
where we see that $V$ is smooth for $X\not=0$.  Here we have to add
the term $\frac {|X|}2$ in the definition of $V$, in order to cancel
the term $\nabla V(\cdot,0)$. Indeed, remember that $U_r (t,0) = -
\frac12$.  If we do not add that term, this would make appear a bad
term like $\frac 1 X$ in the coefficient of the PDE satisfied by $V$
which would not allow us to control the regularity of the solution up
to $X=0$.

\paragraph{Step 1: Estimate on $D^2 V$.}
We make the following pointwise computation of the second derivatives
\begin{align*}
D^2_{ji}V=&D_j(D_i V)=D_j\left(\frac {X_i}{|X|} U_r+\frac 12\frac {X_i}{|X|}\right)\\
=&U_{rr} \frac {X_iX_j}{|X|^2}+\left(U_r+\frac 12\right)
\left(\frac {\delta _{ij}}{|X|}-\frac {X_iX_j}{|X|^3}\right).
\end{align*}
For $R>0$ fixed and $0< r\le R$, we deduce from Lemma \ref{lem:1} that
there exists a constant $C_R>0$ such that
\[|U_{rr}|\le C_R, \quad |U_r+\frac 12|\le r C_R.\]
This implies that $D^2 V\in L^\infty((0,+\infty)\times
(B_R\setminus \{0\}))$. 

Moreover for all $\phi\in C^\infty_c((0,T)\times B_R)$, we have in the distribution sense
\begin{align*}
-\langle D^2_{ji}V, \phi \rangle &=\lim_{\e\to 0}\int_{(0,T)\times (B_R\backslash
  B_\e)} (D_j V) (D_i \phi) \\
&= \lim_{\e\to 0}\left\{\int_{(0,T)\times (B_R\backslash B_\e)}
-(D^2_{ji} V)\phi +\int_{(0,T)\times \partial B_\e}\phi (n\cdot e_i)
D_j V \right\}
\end{align*}
where $n$ is the outward nomal to $B_R\backslash B_\varepsilon$ on the
boundary $\partial B_\varepsilon$, and $e_i$ is a unit vector of the
canonical basis of $\R^3$.  Since $\nabla V$ is bounded, we recover
that
\[(D^2V)_{j,i}=U_{rr} \frac {X_iX_j}{|X|^2}
+\left(U_r+\frac 12\right)\left(\frac {\delta _{ij}}{|X|}-\frac
{X_iX_j}{|X|^3}\right)\] in the distribution sense on
$(0,+\infty)\times B_R$.  This implies that the distribution $D^2 V$
satisfies $D^2 V\in L^\infty((0,+\infty)\times B_R)$.

\paragraph{Step 2: Estimate on $\nabla V$.}
Moreover, since $V_r, V_t\in L^\infty((0,+\infty)\times B_R)$, we get
that for every $\delta>0$ and for every $1<p<+\infty$, there exists a
constant $C=C(\delta, R, p)>0$ such that for every $T\ge \delta$, we
have
\[\|V-V(T,0)\|_{ W^{1,2;p}((T-\delta, T+\delta)\times B_R)}\le C.\]

Using parabolic Sobolev Embedding in parabolic H\"{o}lder spaces (see
\cite[Lemma 3.3]{LSU}), we get, for every $0<\alpha<1$ and a suitable
constant $C=C(\delta, R, \alpha)>0$, that
\begin{equation}\label{eq::ob39}
\|V-V(T,0)\|_{ C^{\frac{1+\a}2,1+\a}_{t,X}((T-\delta, T+\delta)\times B_R)}\le C
\end{equation}
which implies that 
\begin{equation}\label{eq::ob37}
\|\nabla V\|_{C^{\frac{\a}2,\a}_{t,X}((T-\delta, T+\delta)\times B_R)}\le C.
\end{equation}

\paragraph{Step 3: Equation satisfied by $V$.}
A computation gives that $V$ is solution (at least in the
distributional sense) of
\[V_t = A(X, \nabla V)\Delta V + B(X,\nabla V) 
\quad \mbox{for}\quad (t,x)\in (0,T)\times B_{R}(0)\] where
\begin{eqnarray*}
A(X,p) & =&\frac{1}{1+(X\cdot p -\frac{|X|}{2})^2}, \\
B(X,p) &=&\frac{(X\cdot p)}{|X|^2}\frac{q^2}{1+q^2} + \frac{q^2}{|X|}
G_0(q)
\end{eqnarray*} 
with $q=X\cdot p -\frac{|X|}{2}$ and
\[G_0(q)=\frac{1}{q^2}\left(\sqrt{1+q^2}
-\frac12 \left(\frac{2+q^2}{1+q^2}\right)\right)=1+ O(q^2).\]
Let us set
\[\tilde{X}=|X|^\alpha \frac{X}{|X|} \quad \mbox{with}\quad \alpha =1/3.\]
In particular, we can easily check that the map $X\mapsto \tilde{X}$
is in $C^\alpha$ (see Lemma~\ref{lem::ob35}).  Then we can write
\[B(X,p)=(\tilde{X}\cdot p)\frac{(\tilde{X}\cdot p -\frac{|\tilde{X}|}{2})^2}
{1+q^2} + |\tilde{X}| \left(\tilde{X}\cdot p
-\frac{|\tilde{X}|}{2}\right)^2  
G_0(q) \quad \mbox{with}\quad q=X\cdot p -\frac{|X|}{2}.\]
Therefore on the set $\left\{|X|\le R, |p|\le R\right\}$, we see that the function $B$ is Lipschitz continuous both in $p$ and in $\tilde{X}$, i.e. satisfies
\[|B(X',p')-B(X,p)|\le C_R \left(|\tilde{X'}-\tilde{X}| + |p'-p|\right).\]
Using Lemma~\ref{lem::ob35}, this implies (increasing $C_R$ if necessary) that
\[|B(X',p')-B(X,p)|\le C_R \left(|{X'}-{X}|^\alpha + |p'-p|\right)\]
i.e. $B$ is locally Lipschitz in $p$ and $C^\alpha$ in $X$.
Similarly
\[|A(X',p')-A(X,p)|\le C_R \left(|X'-X| + |p'-p|\right)\]
i.e. $A$ is locally Lipschitz in $p$ and  $X$.

Denoting by 
\[\tilde A(t,X)=A(X, \nabla V(t,X))\quad {\rm and}\quad \tilde B(t,X)=B(X, \nabla V(t,X)),\]
and using \eqref{eq::ob37} for the regularity of $\nabla V$, we get
that there exists a constant $C>0$ such that
\begin{equation}\label{eq::ob38}
||\tilde A||_{C^{\frac 16,\frac 13}_{t,X}((T-\delta,T+\delta)\times
  B_{R})}, ||\tilde B||_{C^{\frac 16, \frac
    13}_{t,X}((T-\delta,T+\delta)\times B_{R})} \le C.
\end{equation}
Because $\tilde{V}(t,x):=V(t,x)-V(T,0)$ solves
\[\tilde V_t = \tilde A \Delta \tilde V + \tilde B \quad
\mbox{in}\quad (T-\delta,T+\delta)\times B_{R},\] we can use interior
Schauder estimates (see Proposition \ref{pro::ob36} in the appendix),
and deduce that
\begin{multline*}
\|V-V(T,0)\|_{C^{2+\frac 13, 1+\frac 16}_{x,t}([T, T+\delta]\times
  B_{R/2})}  \\
\le C\left\{||\tilde{B}||_{C_{x,t}^{\alpha,\frac{\alpha}{2}}((T-\delta,T+\delta)\times B_R)} 
 + |V-V(T,0)|_{L^\infty((T-\delta,T+\delta)\times B_R)}\right\}
 \le  C
\end{multline*}
where we have used \eqref{eq::ob38} and \eqref{eq::ob39} for the last
inequality.  This implies in particular \eqref{eq::ob40} (changing
$R/2$ in $R$), and ends the proof of the proposition.
\end{proof}

\section{Existence of a steady state}\label{s3}

The main result of this section is the following proposition.
\begin{pro}[Existence of a steady state]\label{pro::e1}
There exists a constant $\lambda\ge 0$ and a function $\Phi \in
C^\infty(0,+\infty)$, satisfying
\begin{equation}\label{eq::ob24}
-1/2 \le \Phi_r \le 0 \quad \mbox{on}\quad (0,+\infty)
\end{equation}
such that $U(t,r)=\lambda t +\Phi(r)$ is a solution of (\ref{eq:main})
on $\R\times (0,+\infty)$.
\end{pro}
In a first subsection, we build a solution on an annulus $R^{-1}<r< R$,
and in a second subsection we pass to the limit $R\to +\infty$.

\subsection{Steady state in a annulus}

In the following, we will frequently work in log coordinates with the
function $u(t,x)=U(t,e^x)$. The function  $U$ solves
\eqref{eq:main} if and only if $u$ solves the following equation
\begin{equation}\label{eq::mainx}
u_t = F(x,u_x,u_{xx}):=e^{-x}\sqrt{1+u_x^2}+ e^{-2x}u_x +
e^{-2x}\frac{u_{xx}}{1+u_x^2}.
\end{equation}
See for instance \cite{spirale1}.

For $R>1$, we consider the annulus $R^{-1}\le r \le R$, we study the
following problem with Neumann boundary condition on the boundary of
the annulus:
\begin{equation}\label{eq::e2}
\left\{\begin{array}{ll}
\displaystyle rU_t=\sqrt{1 + r^2U_r^2} + U_r \left(\frac{2+ r^2 U_r^2}{1+ r^2
  U_r^2}\right) + \frac{rU_{rr}}{1 + r^2U_r^2} &
\quad \mbox{on}\quad (0,+\infty)\times (R^{-1},R), \medskip\\
U_r=0 &\quad \mbox{on}\quad (0,+\infty)\times \left\{R^{-1},R\right\},
\end{array}\right.
\end{equation}
with initial data
\begin{equation}\label{eq::e2bis}
U(0,r)=U_0(r)  \quad \mbox{for all}\quad r\in [R^{-1},R]
\end{equation}

Then we have the following result.
\begin{lem}[The Cauchy problem in an annulus]\label{lem:grad-estim-r}
Let $R>1$ and $\alpha\in (0,1)$ and assume that $U_0\in
C^{2+\alpha}([R^{-1},R])$ and that $U_0$ satisfies
\begin{equation}\label{eq:condu0}
\left\{\begin{array}{l}
(U_0)_r(R^{-1})=0=(U_0)_r(R)\\
-M\le (U_0)_r(r)\le 0 \quad \mbox{ for all }\quad r\in (R^{-1},R) .
\end{array}\right.
\end{equation}
Then there exists a unique solution $U\in
C^{1+\frac{\alpha}{2},2+\alpha}([0,+\infty)\times [R^{-1},R])$ of
  \eqref{eq::e2}, \eqref{eq::e2bis}.  Moreover $U$ satisfies
\begin{equation}\label{eq::12bis}
  -\max(1/2,M) \le U_r(t,r) \le 0 \quad \mbox{for all}\quad
  (t,r)\in (0,+\infty)\times (R^{-1},R) .
\end{equation}
\end{lem}
\begin{proof}[Proof of Lemma~\ref{lem:grad-estim-r}]
The proof proceeds in several steps.

\paragraph{Step 1: Existence of a smooth solution}
As it is explained in \cite{gik02}, the classical theory allows to
construct a unique solution $U\in C^{2+\alpha,
  1+\frac{\alpha}{2}}([0,+\infty)\times [R^{-1},R])$ of
  \eqref{eq::e2}.
Moreover, from the classical parabolic regularity theory, we can
bootstrap and get that $U\in C^\infty((0,+\infty)\times [R^{-1},R])$.

\paragraph{Step 2: Gradient bound from above.}
We first recall that $u(t,x)=U(t,e^x)$ solves (\ref{eq::mainx}). Let 
$$w=u_x$$
Then by derivation of \eqref{eq::mainx}, we easily get that $w$ solves
in $(0,+\infty) \times (-a,a)$ (with $a=\ln R$),
\begin{equation}\label{eq::e3}
\left.\begin{array}{ll}
w_t = & \displaystyle - e^{-x} \sqrt{1 + w^2} 
+ e^{-x} \frac{ww_x}{\sqrt{1+w^2}} -2e^{-2x}w + e^{-2x}w_x \\ 
      & \displaystyle -2e^{-2x}\frac{w_x}{1+w^2}
+ e^{-2x} \left(\frac{w_{xx}}{1 + w^2}-\frac{2w(w_x)^2}{(1+w^2)^2}\right),
\end{array}\right.
\end{equation}
and 
\[w(t,\pm a)=0  \quad \mbox{for all}\quad t\in (0,+\infty)\]
and
\begin{equation}\label{eq::e3bis}
w(0,x)=e^{x}(U_0)_r(e^x)  \quad \mbox{for all}\quad x\in [-a,a].
\end{equation}
Notice that $\overline{w}=0$ is a supersolution of \eqref{eq::e3}, \eqref{eq::e3bis},
where we use \eqref{eq:condu0} to check the initial condition inequality.
Therefore the classical comparison principle implies that
\begin{equation}\label{eq::e5}
w\le 0.
\end{equation}

\paragraph{Step 3: Gradient bound from below.}
We now define the function
\[z(t,x)=e^{-x}w(t,x)=U_r(t,e^x).\]
It is easy to check that $z$ satisfies
\begin{equation}\label{eq::e16}
z_t = \displaystyle{-  e^{-2x} \left(\frac1{\sqrt{1+w^2}}+z \right)
- e^{-3x} \left(\frac{w}{1 + w^2}+\frac{2w^3}{(1+w^2)^2}\right)
+ e^{-2x}\frac{z_{xx}}{1+w^2} + O(z_x)} .
\end{equation}
Because we already know that $z\le 0$, we deduce that:
\begin{equation}\label{eq::*1000}
e^{2x} z_t\ge  -g(x,z) + \frac{z_{xx}}{1+w^2} + O(z_x)
\end{equation}
with
\[\displaystyle{g(x,z)=   \frac1{\sqrt{1+e^{2x}z^2}}+z 
+  \frac{z}{1 + e^{2x} z^2}}.\]
Let us set
\[h(\gamma)=\gamma + z + \gamma^2 z.\]
Then we have
\[g(x,z)=h(\gamma) \quad \mbox{with}\quad \displaystyle{\gamma= 
\frac1{\sqrt{1+e^{2x}z^2}}} \in (0,1].\] 
Remark that the maximum of
$h(\gamma)$ is reached at $\gamma =-\frac{1}{2z}$ if $z<0$. Therefore
\[\sup_{\gamma\in (0,1]}h(\gamma)\le h\left(-\frac{1}{2z}\right)=z 
- \frac{1}{4z}\le 0  \quad \mbox{if}\quad \displaystyle{z\le -\frac12}.\]
Therefore
\[g(x,z)\le 0 \quad \mbox{if}\quad \displaystyle{z\le -\frac12}.\]
Remark now that $\underline{z}= -\max(1/2,M)$ is then a subsolution of
the equation with equality in (\ref{eq::*1000}) (with zero boundary
conditions).  This implies that $\underline{z}$ is a subsolution of
\eqref{eq::e16} with zero boundary condition.  Again, the comparison
principle for $z$ implies that
\begin{equation}\label{eq::e4}
-\max(1/2,M)\le z.
\end{equation}
Finally \eqref{eq::e4} and \eqref{eq::e5} implies \eqref{eq::12bis}
which ends the proof of the proposition.
\end{proof}
\begin{lem}[Periodic solution in an annulus]\label{lem:constr}
For $R>1$, there exists a solution $U_R$ of \eqref{eq::e2} in
$(0,+\infty) \times (R^{-1},R)$ such that
\begin{equation}\label{eq:per}
U_R (t+ T_R,r) = U_R (t,r) + 2 \pi 
\end{equation}
for some $T_R>0$. 
\end{lem}
\begin{proof}[Proof of Lemma \ref{lem:constr}]
  Let $I$ denote the interval $(R^{-1}, R)$. In view of
  \cite[Remark~2.1]{gik02} and the discussion preceding
  \cite[Proposition~4.3]{gik02}, we know that for all $U_0 \in
  C^{2+\alpha}(\bar I)$ for some $\alpha\in (0,1)$ such that $(U_0)_r
  \le 0$, there exists a solution $U_R$ of \eqref{eq::e2} in
  $(0,\infty) \times I$. Moreover, for all $t >0$, we have
  $U_R(t,\cdot) \in {C}^\infty(\bar{I})$. We then choose $U_0 \in
  C^{2+\alpha} (\bar{I})$ satisfying \eqref{eq:condu0} with $M=1/2$
  and we denote by $U_R$ the corresponding solution. Thanks to
  Lemma~\ref{lem:grad-estim-r}, we know that
\[- 1/2\le (U_R)_r (t,r) \le 0.\]
Moreover, by \cite[Proposition~4.3]{gik02}, there exists a period
$T_R>0$ and $U_0$ such that \eqref{eq:per} holds true.  This achieves
the proof of Lemma~\ref{lem:constr}.
\end{proof}
\begin{lem}[Steady state in an annulus]\label{lem:pertoconst}
For $R>1$, there exists  $\lambda_R > 0$ and $\Phi_R \in {C}^\infty
([R^{-1},R])$ satisfying
\[-1/2 \le (\Phi_R)_r \le 0 \quad \mbox{in}\quad [R^{-1},R] \]
such that $\lambda_R t + \Phi_R (r)$ is a solution of \eqref{eq::e2}. 
\end{lem}
\begin{proof}[Proof of Lemma~\ref{lem:pertoconst}]
  Remark first (with $\lambda_R=\frac {2 \pi }{T_R}$) that 
  \begin{equation}\label{eq:119}
  v(t,r)= U_R (t,r)- \lambda_R t
  \end{equation} is
  $T_R$-periodic with respect to the time variable.  We want to prove
  that it is constant.
 Consider $\eps,\delta >0$ and define
\[
U^{\eps,\delta}_R (t,r) = U_R (t+\delta,r) - \eps.
\]
We have $U^{\eps,0}_R < U_R^{0,0}$ and since $U_R$ is Lipschitz
continuous, $U_R^{\eps,\delta} < U_R^{0,0}$ for $\delta$ small enough. We
then define for $\eps >0$
\[
\delta_\eps = \sup \{ \bar \delta >0 : U_R^{\eps,\delta} < U_R^{0,0}\; \forall \delta\in[0,\bar \delta) \} >0. 
\]
Since $v$ is periodic, we deduce that $\delta_\eps < +\infty$. 
Remark that $U_R^{\eps,\delta}$ and $U_R^{0,0}$ are both solutions of \eqref{eq::e2}
and the optimality of $\delta_\eps$ implies that 
\[
\max_{t \in \left[0,T_R\right], r \in \bar I} \{ U_R^{\eps,\delta_\eps}- U_R^{0,0} \} = 0 .
\]
By Lemma~\ref{lem:hopf} and the Neumann boundary condition, we deduce
that the maximum is attained for some inner point $r_0 \in I$. Since
the function (inside the maximum) is $T_R$-periodic with respect to the time variable, the
strong maximum principle (Theorem~\ref{theo:max-fort}) written for the
difference function $U_R^{\eps,\delta_\eps}- U_R^{0,0}$ implies that
$U_R^{\eps,\delta_\eps} \equiv U_R^{0,0}$ (note that $w=U_R^{\eps,\delta_\eps}- U_R^{0,0}$ solves a linear locally uniformly parabolic equation and the coefficient of the linear equation for $w$ are enough regular to apply the strong maximum principle, see the book  by Gilbarg-Trudinger \cite{GT} for more details on this linearized argument). Then for all $k \in \N$, we
have
\[
U_R (t+ k \delta_\eps, r) = U_R (t,r) + k \eps .
\]
The fact that $U_R-\lambda_RT$ is $T_R$- periodic implies in the limit
$k\to +\infty$ that $\frac{\delta_\eps}{T_R}=\frac
{\eps}{\lambda_RT_R}$, i.e., $\lambda_R=\frac \e{\delta_\eps}$. Hence,
for every $\eps>0$, we have
$$U_R(t+\frac \eps {\lambda_R},r)=U_R(t,r)+\eps.$$
This implies that $v$ defined in \eqref{eq:119} is constant. 
The proof of Lemma~\ref{lem:pertoconst} is now complete.
\end{proof}

\subsection{Steady state in the plane}

In this subsection, we want to take the limit $R\to 0$ to recover a
steady state in the plane. To this end, we first need the
following estimate.
\begin{lem}[Bound on $\lambda_R$]\label{lem::e10}
There exists $\hat{\lambda} \ge 0$ such that for all $R\ge 2$, we have
$0 < \lambda_R \le \hat{\lambda}$.
\end{lem}
\begin{proof}[Proof of Lemma \ref{lem::e10}]
  We already know that $\lambda_R >0$. In order to exhibit
  $\hat{\lambda}\ge 0$ with the desired property, we are going to
  construct a super-solution of \eqref{eq:main} of the type
  $\hat{u}(t,r)=\hat{\lambda} t + \Psi(r)$.

Let $\theta = -\gamma(r)$ describe the circle (in polar coordinates)
of equation $1+\kappa=0$ which is tangent from above to the horizontal
axis.  From an analytical point of view, the reader can check that the
right half circle (i.e. for $\theta \in \left[0,\pi/2\right]$ and
$0\le r\le 2$) corresponds to
\[
\gamma (r) = - \arcsin \left( \frac{r}2 \right)
\]
which satisfies $\bar{F} (r, \gamma_r,\gamma_{rr}) = 0$ for $0<r<2$, where
$$\bar{F} (r, \gamma_r,\gamma_{rr}) := \frac{1}{r}\left\{\sqrt{1 +
  r^2\gamma_r^2} + \gamma_r \left(\frac{2+ r^2 \gamma_r^2}{1+ r^2
  \gamma_r^2}\right) + \frac{r\gamma_{rr}}{1 +
  r^2\gamma_r^2}\right\}.$$

We choose $\Psi$ as follows
\[\Psi(r)=\zeta(r) \gamma(r)\]
where $\zeta$ is a smooth cut-off function which is 
equal to $1$ in $\displaystyle [0,\frac12]$ and equal to zero for $r\ge 1$. 

Now we choose $\hat{\lambda}$ such that
\[
\hat{\lambda} \ge \sup_{r>0}  \bar{F} (r,\Psi_r, \Psi_{rr}) =
\sup_{r \in [1/2,1]} \bar{F} (r,\Psi_r, \Psi_{rr})
\]
We also have for $R\ge 2$:
\[\Psi_r(R)=0,\quad \mbox{and}\quad \Psi_r(R^{-1})=\gamma_r(R^{-1})<0.\]
This implies that $\hat{\lambda} t + \Psi(r)$ is a supersolution of
(\ref{eq::e2}), and the comparison principle with $\lambda_R t +
\Phi_R(r)$ implies (for large times) that $\lambda_R\le \hat{\lambda}$
which ends the proof of the lemma.
\end{proof}

We now want to pass to the limit as $R \to +\infty$ and prove
Proposition \ref{pro::e1}.
\begin{proof}[Proof of Proposition \ref{pro::e1}]
Because the functions $\lambda_R t +\Phi_R(r)$ are uniformly Lipschitz
continuous in space and time independently on $R\ge 2$, we can pass to
the limit $R \to \infty$.  We call the limit $\lambda t + \Phi(r)$,
which is then a viscosity solution of (\ref{eq:main}) and satisfies:
\[-1/2 \le \Phi_r \le 0 \quad \mbox{and}\quad \hat{\lambda}\ge \lambda \ge 0.\]
Because $\lambda t + \Phi(r)$ is globally Lipschitz continuous in space and time,
we can apply Lemma~\ref{lem:1} and deduce that $\Phi\in C^\infty(0,+\infty)$.
This ends the proof of the proposition.
\end{proof}

\section{Asymptotics of the steady state and uniqueness}\label{s4}

The main result of this section is the following proposition.
\begin{pro}[Asymptotics of the steady state and uniqueness]\label{pro::e10}
Assume that $\lambda t +\Phi(r)$ is a globally Lipschitz continuous solution of
\eqref{eq:main} in $\R\times (0,+\infty)$ with $\Phi \in C^\infty
(0,+\infty)$  satisfying
\begin{equation}\label{eq::e11}
\lambda\ge 0 \quad \mbox{and}\quad \Phi_r \le 0.
\end{equation}
Then such a $\lambda$ is unique and such a $\Phi$ is unique up to
an additive constant.  Moreover we have $\lambda >0$ and there exist
constants $a\in\R$ and $C>0$ such that
\begin{equation}\label{eq::e12}
|\Phi(r)+\lambda r + \lambda \ln (1+r) -a | \le \frac{C}{1+r}.
\end{equation}
\end{pro}
We will do the proof of Proposition~\ref{pro::e10} using several
lemmas and propositions.  

\subsection{Positivity of the angular velocity}

We first prove that $\lambda$ is positive. 
\begin{lem}[Positivity of $\lambda$]\label{lem::e13}
Under the assumptions of Proposition~\ref{pro::e10}, we have $\lambda>0$.
\end{lem}
\begin{proof}[Proof of Lemma \ref{lem::e13}]
Assume by contradiction that $\lambda=0$. We look for a barrier
solution that we will compare to $\Phi$.  To this end, let us consider
 the circle in $\R^2$ of radius $1$ (given by the equation
$1+\kappa=0$) and of center $(0,-R)$ for some $R>1$  in the Cartesian
coordinates $X=(x_1,x_2)$.  We can parametrize in polar coordinates,
the right half circle as follows,
\[\theta = -\gamma_R(r):= \arcsin \left(f(r)\right) \quad
\mbox{for}\quad R-1 \le r\le R+1\]
with
\[f(r)=\frac{r^2+R^2 -1}{2Rr}\]
which satisfies $f(R-1)=1=f(R+1)$, and $f'(r)=\frac{r^2-(R^2-1)}{2Rr^2}$
with $f'(R\pm 1)\not=0$.

This implies in particular that the graph of $\gamma_R$ has vertical
tangents at $r=R\pm 1$.  Because $\gamma_R$ is a stationary solution
of (\ref{eq:main}) on $(R-1,R+1)$, we can compare it on $(R-1,R+1)$ to
the stationary solution $\Phi$ when $\lambda=0$.  We consider
$$\min_{r\in [R-1,R+1]} \left(\Phi(r)-\gamma_R(r)\right).$$
Since $\lim_{r\to R-1}\gamma_R'(r)=+\infty$ and $\lim_{r\to
  R+1}\gamma_R'(r)=-\infty$ and using the fact that $\phi$ is
Lipschitz continuous, we get that $(\phi-\gamma_R)(R-1)$ and
$(\phi-\gamma_R)(R+1)$ are local maximum. Hence, the minimum can not
be achieved at $r=R\pm 1$  and is therefore reached at
  some interior point. The strong minimum principle then
implies that
$$\Phi(r)-\gamma_R(r) \text{ is constant in }  (R-1,R+1).$$
By continuity, this is still true at $r=R\pm 1$ which is again impossible.
Finally, we conclude that $\lambda\not =0$ and then $\lambda>0$.
This ends the proof of the lemma.
\end{proof}

\subsection{Asymptotics}

In the following proposition, the asymptotics of the profile is stated
in Log coordinates. It also contains the asymptotics of the derivative
of the profile which will be used later. 
\begin{pro}[Asymptotics near $r=+\infty$]\label{pro::e14}
Under the assumptions of Proposition~\ref{pro::e10}, 
 the function $\varphi(x)=\Phi(e^x)$ satisfies 
\begin{equation}\label{eq::e15}
|\varphi(x)+ \lambda e^x +\lambda x-a|\le Ce^{-x} \quad
\mbox{for}\quad x\ge x_1
\end{equation}
and
\begin{equation}\label{eq::e20}
\varphi_x(x)=-\lambda e^x -\lambda +O(e^{-x})
\quad \mbox{for}\quad x\ge x_1 
\end{equation}
for some constants $a,x_1\in\R$ and $C>0$.
\end{pro}
Recalling \eqref{eq::mainx}, we see that $\varphi$ is a solution of
the following second order ODE
\begin{equation}\label{eq:ode-profile}
\lambda =  e^{-x} \sqrt{1 + \varphi_x^2} + e^{-2x}\varphi_x
+e^{-2x} \frac{\varphi_{xx}}{1 + \varphi_x^2} \quad \mbox{for}\quad x\in  \R .
\end{equation}

As we shall see it, Proposition~\ref{pro::e14} is a consequence of the
study of the ODE satisfied by $v:=\varphi_x\le 0$, which is the
following
\begin{equation}\label{eq::20}
v_x= f(v,x) \quad \text{for} \quad x \in \R 
\end{equation}
where
\[f(w,x)=e^{2x}(1+w^2) \zeta(w,x) 
\quad \mbox{with}\quad \zeta(w,x)=\lambda - e^{-x} \sqrt{1 + w^2} -
e^{-2x}w .\]

We first need the following result.
\begin{lem}[Elementary estimates]\label{lem::21}
Let $\lambda>0$. Then there exists a real number $x_0\ge 0$ such that
for $x\ge x_0$, the equation $f(w,x)=0$ has a single root $w=v_0(x)$
which is non-positive. This function satisfies for $x\ge x_0$
\begin{eqnarray}\label{eq:v0}
  v_0(x)&=&-\lambda e^x -\lambda 
+e^{-x}\left(\frac{1}{2\lambda}-\lambda\right)+O(e^{-2x}),\\
\label{eq:v0x}
(v_0)_x(x)&=&-\lambda e^x + O(1) \le 0.
\end{eqnarray}
Moreover we have
\begin{equation}\label{eq:dfdw}
\frac{\partial f}{\partial w}(w,x)\ge \frac{\lambda^2}{2} e^{3x} \quad
\mbox{for}\quad w\le v_0(x) \quad \mbox{and}\quad x\ge x_0
\end{equation}
and for all $w_*,y_* \in \R$, we have
\begin{equation}\label{eq::28}
\left.\begin{array}{l}
x\ge y_*\ge x_0\\
v_0(y_*)<w_*\le w\le 0
\end{array}\right\} \Longrightarrow  f(w,x)\ge e^{2y_*}\min(\zeta(w_*,y_*),\lambda/2)>0 .
\end{equation}
\end{lem}
\begin{proof}[Proof of Lemma \ref{lem::21}]
The proof proceeds in several steps.

\paragraph{Step 1: Definition of $v_0$.}
Remark that if $f(w,x)=0$, then $w$ solves the following second order
polynomial equation
\begin{equation}\label{eq::23}
(1-e^{-2x}) w^2 + 2\lambda w + 1-\lambda^2e^{2x}=0 \, .
\end{equation}
For some $x$ large enough, there is only one non-positive
solution which is given by the following formula
\begin{align*}
v_0(x)=&\frac{-\lambda 
  -\sqrt{\lambda^2+(1-e^{-2x})(e^{2x} \lambda^2-1)}}{1-e^{-2x}}\\
  =&\frac{-\lambda-\lambda e^x\sqrt{1+\left(\frac 1 {\lambda^2}e^{-4x} -\frac 1{\lambda^2} e^{-2x}\right)}}{1-e^{-2x}}\\
  =& \left(-\lambda-\lambda e^x\left(1-\frac 1{2\lambda^2}e^{-2x}\right)\right)\left(1+e^{-2x}\right)+O(e^{-2x}),
  \end{align*}
  which gives \eqref{eq:v0}.\medskip

In order to recover (\ref{eq:v0x}), we take the $x$-derivative of
equation~(\ref{eq::23}) satisfied by $v_0$, and we get
\[(v_0)_x(v_0(1-e^{-2x}) +\lambda) +(v_0)^2 e^{-2x}  -\lambda^2 e^{2x}=0.\]
This implies (using \eqref{eq:v0} in the second equality)
$$
(v_0)_x(x)=\frac{-(v_0)^2 e^{-2x} +\lambda^2 e^{2x}}{v_0(1-e^{-2x}) +\lambda}
=\frac {\lambda^2e^{2x}+O(1)}{-\lambda e^x+O(1)},
$$
which gives \eqref{eq:v0x}.

\paragraph{Step 2: Estimate on $\frac{\partial f}{\partial w}$.}
Let us now compute
\[\frac{\partial f}{\partial w}(w,x)=2we^{2x} \zeta(w,x) 
+ e^{2x}(1+w^2)\frac{\partial \zeta}{\partial w}(w,x)\]
and
\begin{equation}\label{eq::27}
\frac{\partial \zeta}{\partial w}(w,x)=-\frac{we^{-x}}{\sqrt{1+w^2}}-e^{-2x}=:g(w,x) .
\end{equation}
Remark also that,  increasing $x_0$ if necessary, we have for $x\ge x_0$ both
\[v_0(x)\le -1\] 
and
\[\frac{\partial \zeta}{\partial w}(w,x) \ge \frac12 e^{-x} \quad
\mbox{for}\quad w(x)\le v_0(x)\le -1 .\]
But $\zeta(v_0(x),x)=0$, and then the sign of $\frac{\partial
  \zeta}{\partial w}$ implies
\[\zeta(w(x),x)\le 0 \quad \mbox{for}\quad w(x)\le v_0(x)\]
and 
\[\frac{\partial f}{\partial w}(w,x)\ge e^{2x}(1+w^2)\frac{\partial
  \zeta}{\partial w}(w,x).\]
Again, increasing $x_0$ if necessary,  we can assume that $v_0(x)\le -\lambda
e^{x}$ for $x\ge x_0$ and then
\[\frac{\partial f}{\partial w}(w,x)\ge \frac{\lambda^2}{2} e^{3x} \quad
\mbox{for}\quad w(x)\le v_0(x).\]

\paragraph{Step 3: Estimate on $f$.}
Recall that the function $g$ appears in \eqref{eq::27}.  Remark that
for $x\ge 0$ we have $g(w,x)=0$ with $w\le 0$ if and only if
\[w(x)=-\frac{1}{\sqrt{e^{2x}-1}}=:w_0(x).\]
Moreover we can then deduce that
\begin{eqnarray*}
g(w,x)&\ge& 0 \quad \mbox{if}\quad w\le w_0(x), \\
g(w,x)&\le& 0 \quad \mbox{if}\quad w_0(x)\le w\le 0.
\end{eqnarray*}
Because of \eqref{eq::27}, we deduce that, increasing $x_0$ if necessary,
\[w_0(x)\le w\le 0 \Longrightarrow \zeta(w,x) \ge \zeta(0,x)=\lambda -
e^{-x}\ge \lambda/2 \quad \mbox{if}\quad x\ge x_0\]
and then using the definition of $f$ and a bound from below of
$\zeta(w,x)$ for $w\in \left[w_*,0\right]$, we get
\[v_0(x)<w_*\le w\le 0 \Longrightarrow f(w,x) \ge
e^{2x}\min(\zeta(w_*,x),\lambda/2)>0 \quad \mbox{if}\quad x\ge x_0\]
Let us notice that for $w\le 0$, we have up to increase $x_0$ if necessary, 
\[\frac{\partial \zeta}{\partial x}(w,x)=e^{-x} \sqrt{1 + w^2}
+2e^{-2x}w \ge 0 \quad \mbox{if}\quad x\ge x_0\]
and then this implies (\ref{eq::28}). This ends the proof of the lemma.
\end{proof}
We next prove the following estimate. 
\begin{lem}[Asymptotics for $v=\varphi_x$]\label{lem:asympt-v}
For any $\mu>0$, there exists a real number $x_1\ge x_0$ such that
$v=\varphi_x$ satisfies
\[v_0(x)\ge v(x)\ge v_0(x) -\mu e^{-\frac32 x}\quad
\mbox{for}\quad x\ge x_1\]
where $v_0$ and $x_0$ are given by Lemma~\ref{lem::21}.
\end{lem}
\begin{proof}[Proof of Lemma \ref{lem:asympt-v}]
Recall that $\lambda >0$ and define
\[\bar{v}(x):= v_0(x) -\mu e^{-\frac32 x}.\] 
The proof proceeds in several steps.

\paragraph{Step 1: $\mathbf{\bar{v}}$ is a super-solution.} 
Remark that, thanks to \eqref{eq:v0x},
\[\bar{v}_x(x)=(v_0)_x(x)+\frac32\mu e^{-\frac32 x} = -\lambda e^{x} + O(1).\]
We also remark that there exists $w(x)\in [\bar{v}(x),v_0(x)]$ such that
\begin{equation}\label{eq::25}
\begin{array}{ll}
f(\bar{v}(x),x)& \displaystyle{=f(v_0(x),x)
+ \frac{\partial f}{\partial w}(w(x),x)(\bar{v}(x)-v_0(x))}\medskip\\
& \le \displaystyle{\frac{\lambda^2}{2} e^{3x}(\bar{v}(x)-v_0(x))} \medskip\\
& \le \displaystyle{-\mu \frac{\lambda^2}{2} e^{\frac32 x}} 
\end{array}
\end{equation}
where we used \eqref{eq:dfdw} in the second line. 
Therefore there exists $x_1\ge x_0$ such that
\[\bar{v}_x(x) \ge f(\bar{v}(x),x) \quad \mbox{for}\quad x\ge x_1.\]

\paragraph{Step 2: Comparison with $\mathbf{\bar v}$.}
Assume by contradiction that $v(x_*)\le \bar{v}(x_*)$ for some $x_*\ge
x_1$.  Then, from the comparison principle, we deduce that
\[v(x)\le \bar{v}(x)\quad \mbox{ for all }\quad x\ge x_* .\]
Then we have
\begin{equation}\label{eq::26}
v_x(x)=f(v(x),x) \le f(\bar v(x),x)\le \displaystyle{-\mu \frac{\lambda^2}{2} e^{\frac32 x}}
\end{equation}
where we have used the fact that $v\le \bar v$, the monotonicity of
$f(w,x)$ in $w$ (see \eqref{eq:dfdw}) and estimate~(\ref{eq::25}).
Estimate~(\ref{eq::26}) now gives a contradiction with the fact that
$\Phi_r(e^x)=e^{-x}v(x)$ is bounded.

\paragraph{Step 3: $\mathbf{v_0}$ is a sub-solution.}
The inequality $(v_0)_x(x)\le 0 =f(v_0(x),x)$ for $x\ge x_0$ follows
from \eqref{eq:v0x}.

\paragraph{Step 4: Comparison with $\mathbf{v_0}$.}
We argue by contradiction.  Let us assume that there exists a point
$y_*\ge x_0$ such that $v(y_*)>v_0(y_*)$.  Then from (\ref{eq::28}),
we deduce that there exists a constant $\alpha>0$ such that
\[f(w,x)\ge \alpha>0 \quad \mbox{for}\quad w \in [v(y_*), 0] 
\quad \text{and} \quad x\ge y_* .\]
But recall that 
\[v_x(x)=f(v(x),x).\]
This implies that 
\[v_x(x)\ge \alpha \quad \mbox{for}\quad x\ge y_*\quad \mbox{while}\quad v(x)\le 0.\]
Therefore we conclude (using the continuity of $f$) that there exists
a point $x_2$ such that $v(x_2)>0$, which is impossible because
$v=\varphi_x\le 0$. We thus get the desired contradiction.  This ends
the proof of the lemma.
\end{proof}

\begin{proof}[Proof of Proposition~\ref{pro::e14}]
It follows from Lemma~\ref{lem:asympt-v} and (\ref{eq:v0}). 
\end{proof}

\subsection{Uniqueness}
\label{s4.2}

\begin{pro}[Uniqueness]\label{pro::31}
Under the assumptions of Proposition~\ref{pro::e10}, $\lambda$ is unique
and $\Phi$ is unique up to addition of constants.
\end{pro}
In order to prove Proposition \ref{pro::31}, we will need the
following space Liouville result which will be proven later in
Section~\ref{s6} as an independent result.
\begin{theo}[Space Liouville theorem]\label{theo::pp1}
Let $\Phi^i$ for $i=1,2$ be two $C^2(\left[0,+\infty\right))$ functions
  such that for some $\lambda>0$, the functions $\lambda t +\Phi^i(
  r)$ are solutions of (\ref{eq:main}) in $\R\times (0,+\infty)$ for
  $i=1,2$.  Assume also that we have for $i=1,2$ and $r\ge 0$:
\begin{equation}\label{eq:prof}
\left|\Phi^i(r) +\lambda r + \lambda \ln (1+r)\right| \le
\frac{C}{1+r}
\end{equation}
and
\[\left|\Phi^i_r(r) +\lambda \right| \le \frac{C}{1+r}.\]
Then $\Phi^1=\Phi^2$.
\end{theo}
\begin{proof}[Proof of Proposition \ref{pro::31}]
We already know that $\Phi$ satisfies \eqref{eq::e12}.
From Proposition~\ref{pro::ob30}, we deduce that $\Phi\in
C^2([0,+\infty))$.

\paragraph{Uniqueness of $\lambda$.}
We argue by contradiction by assuming that there exist
  $(\Phi^1,\lambda^1)$ and $(\Phi^2,\lambda^2)$ two solutions
  such that
\[\lambda^1< \lambda^2 .\]
Because of \eqref{eq::e12}, we deduce that there exists a
constant $K$ such that
\[\Phi^1(r) \ge \Phi^2(r) -K \quad \mbox{for}\quad r > 0.\]
From the comparison principle for \eqref{eq:main} (see Theorem 1.3 in
\cite{spirale1}, with Lipschitz continuous initial data $U_0=\Phi^1$), we deduce
\[\lambda^1t + \Phi^1(r) \ge \lambda^2 t + \Phi^2(r) 
-K \quad \mbox{for all}\quad (t,r)\in (0,+\infty)\times (0,+\infty)\]
which implies (for large times) that  $\lambda^1\ge \lambda^2$. This is the desired
contradiction.

\paragraph{Uniqueness of $\Phi$ (up to an additive constant).}
We now consider two profiles $\Phi^1$, $\Phi^2$ with the same
$\lambda=\lambda^1=\lambda^2$.  Recall that for $i=1,2$, each function
$\Phi^i$ satisfies (\ref{eq::e12}) for some constant $a^i$.  Adding
different constants to those two functions if necessary, we can asssume that
$a^1=a^2=a=0$, i.e.
\[|\Phi^i(r)+\lambda r +\lambda \ln (1+r)|\le\frac {C} {1+r},
\quad \mbox{for}\quad i=1,2.\]
We then apply Theorem~\ref{theo::pp1} to conclude that
$\Phi^1=\Phi^2$.  The proof is now complete.
\end{proof}

\begin{proof}[Proof of Proposition~\ref{pro::e10}]
It follows from Lemma~\ref{lem::e13} and Propositions~\ref{pro::e14}
and \ref{pro::31}.
\end{proof}

\section{Further properties of the steady state}\label{s5}

\subsection{Monotonicity properties}

\begin{pro}[Monotonicity of the gradient of the profile]\label{pro::*14}
Let $\Phi$ be the profile given in Proposition~\ref{pro::e10}. Then we have
\begin{equation}\label{eq::e21}
\Phi_{rr}\ge 0 \quad \mbox{in}\quad [0,+\infty)
\end{equation}
\begin{equation}\label{eq::e33}
-\frac12\le \Phi_r\le -\lambda
\end{equation}
and
\begin{equation}\label{eq::e32}
\Phi_r(0)=-\frac{1}{2} \quad \mbox{and}\quad \Phi_r(+\infty)=-\lambda<0.
\end{equation}
\end{pro}
\begin{proof}[Proof of Proposition \ref{pro::*14}]
For $\varphi(x)=\Phi(e^x)$, we recall from (\ref{eq::e16}) that
\[z(x):= e^{-x} \varphi_x (x)=\Phi_r(e^x)\] 
satisfies with $w=\varphi_x$:
\[0=z_t = \displaystyle{-  e^{-2x} \left(\frac1{\sqrt{1+w^2}}+z \right)
- e^{-3x} \left(\frac{w}{1 + w^2}+\frac{2w^3}{(1+w^2)^2}\right)
+ e^{-2x}\frac{z_{xx}}{1+w^2} + O(z_x)} \, .\]

\paragraph{Step 1: case of a local minimum of $z$.}
Assume that $z$ has a local minimum at $x_0$ with value $z_0=z(x_0)$. Then $z_{xx}(x_0)\ge 0$
and $z_x(x_0)=0$ which implies,
\[\displaystyle \frac{1}{\sqrt{1+e^{2x_0}z_0^2}}+z_0 +\frac{z_0}{1 +
  e^{2x_0}z_0^2}+\frac{2e^{2x_0}z_0^3}{(1+e^{2x_0}z_0^2)^2} \ge 0.\] 
Setting
\[\gamma=\frac1{\sqrt{1+e^{2x_0}z_0^2}}\in (0,1],\]
we see that this means
\[\displaystyle{\gamma+z_0 
+\gamma^2 z_0 +2 \gamma^4 z_0 (1/\gamma^2 -1) 
}\ge 0\]
i.e.
\begin{equation}\label{eq::e17}
\gamma+z_0 (1+3\gamma^2 -2\gamma^4) \ge 0.
\end{equation}
Let
$$g(y):=1+3y -2y^2$$
Remark that $g$ is maximum at $y=3/4$ and then
\[\inf_{y\in (0,1]} g(y) \ge \min\left(g(0),g(1)\right)=1.\]
Therefore (\ref{eq::e17}) means
\[z_0\ge -\frac{\gamma}{g(\gamma^2)}=: -K(\gamma).\]

\paragraph{Step 2: Monotonicity of $K$.}
Let us compute with $y=\gamma^2$:
\[K'(\gamma)=\frac{1}{g^2(y)}{(g(y)-2yg'(y))}\]
with
\[g(y)-2yg'(y)= 1+3y -2y^2 -2y(3-4y)=1-3y+6y^2=:h(y)\]
which is minimal at $y^*=1/4$ with value $h(y^*)>0$.
Therefore $K$ is increasing.

\paragraph{Step 3: Monotonicity of $z$.}
Assume now that $z$ has a local maximum at $\overline{x}$ with value
$\overline{z}=z(\overline{x})$.  Then we have
\[\overline{z} \le -K(\overline{\gamma}) \quad \mbox{with}\quad \overline{\gamma}=
\frac{1}{\sqrt{1+e^{2\overline{x}}\overline{z}^2}}.\]
We already know (see \eqref{eq::e20}) that
\[z(x)=-\lambda -\lambda e^{-x} +o(e^{-2x})\]
which shows that $z$ cannot be non-increasing in
$(\overline{x},+\infty)$ (and satisfies
$z(+\infty)=\Phi_r(+\infty)=-\lambda$).  Therefore there exists
$\underline{x}> \overline{x}$ such that $z$ has a local minimum at
$\underline{x}$ with value $\underline{z}=z(\underline{x})$ that we
can choose such that
\begin{equation}\label{eq::*15}
\underline{z}\le \overline{z}\le 0.
\end{equation} 
Moreover we have
\[\underline{z} \ge -K(\underline{\gamma}) \quad \mbox{with}\quad \underline{\gamma}=
\frac{1}{\sqrt{1+e^{2\underline{x}}\underline{z}^2}}< \overline{\gamma}.\]
The strict monotonicity of $K$ implies
\[\overline{z} \le -K(\overline{\gamma}) < -K(\underline{\gamma}) \le \underline{z},\]
which is in contradiction with (\ref{eq::*15}). Therefore, we conclude
that $z$ has no local maximum.

\paragraph{Step 4: Behaviour at $r=0$.}
We recall that $\Phi\in C^2([0,+\infty))$.  From the fact that
  $\lambda t + \Phi(r)$ is a solution of (\ref{eq:main}), we deduce
  that
\[r\lambda=\sqrt{1 + r^2\Phi_r^2} + \Phi_r \left(\frac{2+ r^2 \Phi_r^2}{1+ r^2
  \Phi_r^2}\right) + \frac{r\Phi_{rr}}{1 + r^2\Phi_r^2}.\]
At $r=0$, we deduce that
\begin{equation}\label{eq::e22}
1+2\Phi_r(0)=0.
\end{equation}
Close to $r=0$, we deduce (by Tayor expansion) that
\[\Phi_{rr}(r)=O(r) + \lambda -\frac{1}{r}\left(1+2\Phi_r(r) + O(r^2)\right).\]
Using (\ref{eq::e22}), we deduce that
\[\Phi_{rr}(0)=\frac{\lambda}{3}>0.\]

\paragraph{Step 5: Conclusion.}
Using the fact that $\Phi_{rr}(0)>0$ and the fact that $\Phi_r$ has no
local maximum (by Step 3), we deduce that $\Phi_r$ is increasing,
which in particular implies (\ref{eq::e21}) and (\ref{eq::e33}).  This
ends the proof of the proposition.
\end{proof}

\begin{pro}[Sign and monotonicity of the curvature]\label{pro::e1000}
Let $\Phi$ be the profile given in Proposition~\ref{pro::e10}. Then the
curvature $\kappa_\Phi$ defined in \eqref{def:courbure} satisfies,
\[-1\le \kappa_\Phi\le 0\]
and
\[\kappa_\Phi(0) =-1,\quad \kappa_\Phi(+\infty)=0.\]
Moreover we have
\[(\kappa_\Phi)_r \ge 0.\]
\end{pro}
\begin{proof}[Proof of Proposition \ref{pro::e1000}]
We set $\kappa(x):=\kappa_\Phi(e^x)$. Notice that we deduce from 
\eqref{def:courbure} and \eqref{eq::e32} that
\[\kappa_\Phi(r=0)=2\Phi_r(0)=-1.\]

\paragraph{Step 1: $\mathbf{\kappa \in [-1,0]}$.}
Recall that for the profile, we have,
\begin{equation}\label{eq::1001}
\lambda = e^{-x}\sqrt{1+u_x^2} + e^{-2x}u_x + e^{-2x}
\frac{u_{xx}}{1+u_x^2}= e^{-x}(1+\kappa)\sqrt{1+u_x^2}
\end{equation}
where the curvature $\kappa$ can be written as
\begin{equation}\label{eq::1000}
\kappa := e^{-x}\frac{u_x}{\sqrt{1+u_x^2}} + e^{-x} \frac{u_{xx}}{(1+u_x^2)^{\frac32}}.
\end{equation}
Equation~\eqref{eq::1001} shows that we can find the following other
expression for the curvature,
\begin{equation}\label{eq::e30}
\kappa= \frac{\lambda e^x}{\sqrt{1+u_x^2}} -1.
\end{equation}
Using  (\ref{eq::e20}), we then deduce that 
\begin{equation}\label{eq:001}
\kappa(x=+\infty)=0.
\end{equation}
Moreover, using again \eqref{eq::e30}, we have
\begin{align*}
\kappa_x &\displaystyle =  \frac{\lambda e^x}{\sqrt{1+u_x^2}} 
-\frac{\lambda e^x}{(1+u_x^2)^{\frac32}}u_xu_{xx}\medskip\\
&\displaystyle =\frac{\lambda e^x}{\sqrt{1+u_x^2}}  
-\lambda e^{2x} u_x \left(\kappa -e^{-x}\frac{u_x}{\sqrt{1+u_x^2}}\right)\medskip\\
&\displaystyle =\lambda e^x\sqrt{1+u_x^2} -\lambda e^{2x} u_x  \kappa.
\end{align*}
Using the fact that $u_x\le 0$, we conclude that 
\[\kappa(x_0) > 0 \quad \Longrightarrow \quad \kappa(x) \ge
\kappa(x_0) \quad \mbox{for}\quad x\ge x_0\]
which is in contradiction with (\ref{eq:001}). Therefore $\kappa\le 0$.
The fact that $1+\kappa \ge 0$ comes directly from (\ref{eq::e30}).

\paragraph{Step 2: $\kappa$ is non-decreasing.}
Let us start again from
\begin{equation}\label{eq::e31}
\kappa_x= \lambda e^x\sqrt{1+u_x^2} -\lambda e^{2x} u_x  \kappa.
\end{equation}
Then
\begin{align*}
\kappa_{xx} & \displaystyle =\lambda e^x\sqrt{1+u_x^2} -2\lambda e^{2x} u_x  \kappa
+ \lambda e^x\frac{u_xu_{xx}}{\sqrt{1+u_x^2}} -\lambda e^{2x} u_{xx}
\kappa -\lambda e^{2x} u_x  \kappa_x\\
&\displaystyle = 2\kappa_x -\lambda e^x\sqrt{1+u_x^2}  -\lambda e^{2x}
u_x  \kappa_x +\frac{u_{xx}}{u_x}\left(\kappa_x -\frac{\lambda e^x}{\sqrt{1+u_x^2}}\right)\\
&\displaystyle = \kappa_x \left(2-\lambda e^{2x} u_x
+\frac{u_{xx}}{u_x}\right) -\frac{\lambda e^x \sqrt{1+u_x^2}}{u_x}\left( u_x
+ \frac{u_{xx}}{1+u_x^2}\right)\\
&\displaystyle = \kappa_x \left(2-\lambda e^{2x} u_x 
+\frac{u_{xx}}{u_x}\right) -\frac{\lambda e^x \sqrt{1+u_x^2}}{u_x}e^x \sqrt{1+u_x^2} \kappa.
\end{align*}
Recall that $u_x<0$, $\kappa\le 0$ and $\kappa_x=0$ implies in
(\ref{eq::e31}) that $u_x \kappa =e^{-x}\sqrt{1+u_x^2}>0$, which
shows that $\kappa<0$.  Therefore we conclude from the above
computation that
\[\kappa_{xx} <0 \quad \mbox{at any point where}\quad \kappa_x=0.\]
This implies that $\kappa$ can not have local minima.  Because $-1\le
\kappa(x)\le 0$ and $\kappa(-\infty )=-1$, $\kappa(+\infty)=0$, we
deduce that $\kappa$ does not have local maxima neither (which would
imply the existence of a local minimum).  Therefore
\[\kappa_x \ge 0.\]
This ends the proof of the proposition.
\end{proof}

\subsection{Bound from below for the angular velocity}

We next prove the following lemma.
\begin{lem}[Bound from below on $\lambda$]\label{lem::1003}
We have $\lambda \ge 1/4$.
\end{lem}
\begin{proof}[Proof of Lemma \ref{lem::1003}]
The proof proceeds in several steps. 

\paragraph{Step 1: comparison.}
The idea is to revisit the proof of the uniqueness of $\lambda$.  For
some $\mu>0$, we set
\[\varphi_1:=\varphi\quad \mbox{and}\quad \varphi_2:=-\mu e^x.\]
If 
\[\mu> \lambda,\]
then a comparison of the behaviour at $x=+\infty$ implies that
\[\varphi_2 \le \varphi_1 +K \quad \mbox{on}\quad \R\]
for some suitable constant $K$.
We recall that
\[\lambda =F(x,\varphi_x,\varphi_{xx}),\]
with $F$ defined in \eqref{eq::mainx}. We then define
\[\begin{array}{ll}
h_\mu(x) &:=F(x,(\varphi_2)_x,(\varphi_2)_{xx})\\
\\
& \displaystyle = \sqrt{e^{-2x}+\mu^2} -\mu e^{-x}\left(1+\frac{1}{1+\mu^2 e^{2x}}\right).
\end{array}\]
If 
\begin{equation}\label{eq::1004}
\alpha \le \inf_{x\in\R} h_\mu(x),
\end{equation}
then we can take $\lambda_2=\alpha$ and we see with $\lambda_1:=\lambda$ that
\[\lambda_2 t +\varphi_2(x) \le \lambda_1 t +\varphi_1(x) +K\]
is true at $t=0$ and then is true for every time $t\ge 0$, because the
left hand side is a subsolution and the right hand side is a
solution. Then we conclude that
\[\alpha= \lambda^2 \le \lambda^1 =\lambda\]
i.e.
\begin{equation}\label{eq::1005}
\mu>\lambda \quad \Longrightarrow \quad \lambda \ge \alpha \quad
\mbox{if}\quad \alpha \mbox{ satisfies (\ref{eq::1004}).}
\end{equation}

\paragraph{Step 2: estimate on $\alpha$ and conclusion.}
Remark that (\ref{eq::1004}) is satisfied for $\alpha\ge 0$ if and only if
\begin{equation}\label{eq::pp1}
\left(\alpha + \mu e^{-x}\left(1+\frac{1}{1+\mu^2 e^{2x}}\right) \right)^2 \le e^{-2x}+\mu^2
\end{equation}
Because we have
\[  \left(\alpha + \mu e^{-x}\left(1+\frac{1}{1+\mu^2 e^{2x}}\right) \right)^2\le
2\alpha^2 + 2e^{-2x} \mu^2 2^2\] 
we see that inequality \eqref{eq::pp1} is satisfied in particular
if
\[2\alpha^2 \le \mu^2 \quad \mbox{and}\quad 8\mu^2 \le 1.\]
For instance for 
\[\mu=1/(2\sqrt{2}) \quad \mbox{and}\quad \alpha = 1/4,\]
we conclude from (\ref{eq::1005}) that $\lambda\ge 1/4$.
This ends the proof of the lemma.
\end{proof}
\begin{proof}[Proof of Theorem~\ref{th:steady}]
Apart from \eqref{eq:profil courbure}, Theorem~\ref{th:steady} is then
a consequence of Propositions~\ref{pro::e1}, \ref{pro::e10},
\ref{pro::*14}, \ref{pro::e1000} and Lemma~\ref{lem::1003}.
As far as \eqref{eq:profil courbure} is concerned, it is a simple
consequence of 
\[0 \le 1 + \kappa_\Phi = \frac{r\lambda}{\sqrt{1+r^2 \Phi_r^2}}.\]
The proof of Theorem~\ref{th:steady} is now complete.
\end{proof}

\section{A Liouville result}\label{s6}

This section is devoted to the proof of a Liouville result (Theorem
\ref{theo:liouville}) for global solutions of \eqref{eq:main}.
This Liouville result will be used in the next section. The Liouville
Theorem~\ref{theo:liouville} classifies global space-time
solutions. Such kind of results have been for instance obtained for
certain nonlinear heat equations in \cite{GK, MZ}, where the nonlinearity
comes from the source term. On the contrary, the nonlinearity in our
problem comes from the geometry itself.

In order to prove Theorem~\ref{theo:liouville}, we first prove two
comparison principles: one for small $r$'s (i.e. in $\R\times
[0,r_0^-)$), and one for large $r$'s (i.e. in $\R\times
  [r_0^+,+\infty)$).
\begin{pro}[Comparison principle for small $r$'s]\label{prop:comp-r-petit}
Given some constant $C>0$, there exists some $r_0^-=r_0^-(C)>0$ such
that the following holds for every $r_0\in (0,r_0^-]$.  Let $U\in
  C^{2,1}(\R\times [0,r_0])$ be a subsolution and $V\in
  C^{2,1}(\R\times [0,r_0])$ be a supersolution of \eqref{eq:main} in
  $\R \times (0,r_0)$ satisfying
\begin{equation}\label{eq::ob9}
1+2V_r(t,0)\le 0 \le 1 + 2U_r(t,0)  \quad \mbox{for all}\quad t\in\R.
\end{equation}
Assume moreover that we have
\begin{equation}\label{eq::ob2}
\left\{\begin{array}{l}
|U_r|, |V_r|\le C,\\
|rU_{rr}|\le C,\\
|U-V|\le C.
\end{array}\right.
\end{equation}
If $U \le V$ in $\R\times \left\{r_0\right\}$, then $U \le V$ in $\R
\times [0,r_0]$.
\end{pro}
\begin{rem}[The Neumann boundary condition]
Notice that condition (\ref{eq::ob9}) can be seen as the
evaluation on the boundary $r=0$ of the inequalities in equation
(\ref{eq:main}) associated to subsolutions $U$ and supersolutions $V$.
\end{rem}
\begin{proof}[Proof of Proposition \ref{prop:comp-r-petit}]
The proof proceeds in several steps. 

\paragraph{Step 1: subsolution $W=U-V$.}
We set $W=U-V$. We write the difference of the two inequalities
satisfied by $U$ and $V$, which gives
\begin{multline*}
rW_t\le G(rU_r)-G(rV_r) + U_r(K(rU_r)-K(rV_r)) \\
+ (U_r-V_r)K(rV_r) + rU_{rr}(H(rU_r)-H(rV_r)) + r H(rV_r) W_{rr}
\end{multline*}
with
\begin{equation}\label{eq::ob21}
G(p)=\sqrt{1+p^2},\quad K(p)=\frac{2+p^2}{1+p^2},\quad H(p)=\frac{1}{1+p^2}.
\end{equation}
This leads to
\begin{equation}\label{eq::ob5}
W_t \le A W_r  + H(rV_r) W_{rr} \quad \mbox{on}\quad \R\times (0,r_0)
\end{equation}
with
\begin{equation}\label{eq::ob20}
A= a + \frac{K(rV_r)}{r} + b +c
\end{equation}
where
\begin{equation}\label{eq::ob4}
\left\{\begin{array}{l}
a= \int_0^1 ds\ G'(r(U_r-sW_r)),\\
b=U_r \int_0^1 ds\ K'(r(U_r-sW_r)),\\
c =  rU_{rr} \int_0^1 ds\ H'(r(U_r-sW_r)).
\end{array}\right.
\end{equation}
Using (\ref{eq::ob2}) and the fact that $|G'(p)|\le |p|$ and
$|K'(p)|=|H'(p)|\le 2|p|$, this implies that
\[A\ge -r_0C + \frac{2}{r_0} -2r_0 C^2 -2r_0 C^2.\]
Choosing then $r_0=r_0(C)>0$ small enough, we deduce that
\begin{equation}\label{eq::ob6}
A\ge 0 \quad \mbox{and}\quad H(rV_r)\ge \frac12.
\end{equation}

\paragraph{Step 2: supersolution $\Psi$.}
The goal is now to construct a non-negative supersolution
(i.e. satisfying the reverse inequality in (\ref{eq::ob5})) which
explodes as $|t|\to +\infty$.  We define for some $\mu>0$
\[\Psi(r,t)=e^{-\mu t}\zeta(r)+f(t)\]
with
\[0\le f\in C^\infty (\R)\quad {\rm s.t.}\quad 
\left\{\begin{array}{l}
f(t)=0 \quad {\rm if} \; t<0\\
f'\ge 0\\
f(t)\to+\infty \quad {\rm as}\; t\to +\infty
\end{array}
\right.\]
such that we have
\[\left\{\begin{array}{l}
-\mu \zeta \ge \frac12 \zeta_{rr} \quad \mbox{in}\quad (0,r_0),\\
\zeta_r(0)=0.
\end{array}\right.\]
We can simply choose
$\zeta(r):=\cos\left(\frac{\pi}{4}\frac{r}{r_0}\right)$ with
$2\mu:=\left(\frac{\pi}{4r_0}\right)^2$.  Because $\zeta_r\le 0,
\zeta_{rr}\le 0$ on $(0,r_0)$, we get, using (\ref{eq::ob6}), that
\[\Psi_t \ge \frac12 \Psi_{rr} \ge A\Psi_{r} + H(ru_r)\Psi_{rr} \quad
\mbox{on}\quad \R\times (0,r_0).\]

\paragraph{Step 3: contact point.}
Notice that $\Psi \ge \delta >0$ on $\R\times [0,r_0]$.  Then for
$\varepsilon>0$ large enough, we have:
\[\varepsilon \Psi  \ge W \quad \mbox{on}\quad \R\times [0,r_0].\]
We can then decrease $\varepsilon$ untill we get a contact point,
\[\varepsilon^*=\inf\{\varepsilon \ge 0,\; \varepsilon \Psi \ge W
\quad \mbox{on}\quad \R\times [0,r_0]\}.\]
We now want to show that $\varepsilon^*=0$. By contradiction, assume
that $\varepsilon^*>0$. We have
\begin{equation}\label{eq::ob7}
\inf_{(t,r)\in \R\times [0,r_0]} \{\e^*\Psi -W\}=0.
\end{equation}
Because $W$ is bounded and 
\[\liminf_{|t|\to +\infty} \inf_{r\in [0,r_0]} \Psi(t,r) =+\infty\]
we deduce that the infimum in \eqref{eq::ob7} is reached at some point
$(t^*,r^*)\in \R\times [0,r_0]$.  Because $\varepsilon^*\Psi\ge
\varepsilon^*\delta>0$ and $W\le 0$ for $r=r_0$ we deduce that $r^*\in
           [0,r_0)$.  Recall that
\[\bar W = \varepsilon^* \Psi- W\]
solves
\[\left\{\begin{array}{l}
\left.\begin{array}{l}
\bar W_t \ge A \bar W_r  + H(rV_r) \bar W_{rr},\\
\bar W \ge 0
\end{array}\right|\quad \mbox{on}\quad \R\times (0,r_0),\\
\ \bar W(t^*,r^*)=0,\\
\ \bar W_r(t,0)\le 0 \quad \mbox{for all}\quad t\in\R,
\end{array}\right.\]
and as a consequence of our assumptions, 
the functions $A$ and $H(rV_r)$ are continuous on $\R\times (0,r_0]$.

\paragraph{Case 1: $r^*>0$.}
Then we can apply the strong maximum principle (see
Theorem~\ref{theo:max-fort}) and deduce that
\begin{equation}\label{eq::ob8}
\varepsilon^* \Psi =W \quad \mbox{on}\quad (-\infty,t^*]\times [0,r_0],
\end{equation}
which is absurd for $r=r_0$.

\paragraph{Case 2: $r^*=0$.}
If the coefficient $A$ would have been continuous up to $r=0$, then we
would have applied Hopf lemma (see Lemma~\ref{lem:hopf}) to deduce
again (\ref{eq::ob8}), in order to get the same contradiction.

The difficulty here is that the coefficient $A$ blows-up as $r$ goes
to zero.  We can easily circumvent this difficulty, if we replace
$\Psi$ with
\[\tilde{\Psi}:=\Psi -\eta r\]
for some $\eta >0$ small enough. Now at the point $(t^*,0)$ of minimum
of $\bar W=\varepsilon^*\tilde{\Psi} -W$, we get in particular that
\[0\le \bar W_r(t^*,0)= -\varepsilon^*\eta -W_r(t^*,0).\]
On the other hand, we have by assumption
\[W_r(t^*,0) = (U_r-V_r)(t^*,0) \ge 0\]
which gives a contradiction.  Therefore, in all cases, we conclude
that $\varepsilon^* =0$, which means that $W\le 0$.  This ends the
proof of the proposition.
\end{proof}

\begin{pro}[Comparison principle for large $r$'s]\label{prop:comp-r-grand}
Given some constants $\lambda>0$, $\delta>0$ and $L_0\ge 1$, there
exists $r_0^+ =r_0^+(\delta, L_0,\lambda)>0$ such that the following
holds for all $r_0\in [r_0^+,+\infty)$.  Let $U\in C^{2,1}(\R \times
  [r_0,+\infty))$ be a subsolution and $V\in C^{2,1}(\R \times
    [r_0,+\infty))$ be a supersolution of \eqref{eq:main} on $\R
      \times (r_0,+\infty)$, satisfying in $\R \times [r_0,+\infty)$,
\begin{equation}\label{eq:200}
\left\{\begin{array}{l}
-L_0\le U_r, V_r\le -\delta,\\
|U(t,r)-\lambda t-\Phi_0(r)|\le C,\\
|V(t,r)-\lambda t-\Phi_0(r)|\le C,\\
\left|(\Phi_0)_r(r)\right|\le L_0
\end{array}\right.
\end{equation}
for some function $\Phi_0$ and some constant $C>0$.

 If $U \le V$ on $\R\times \left\{r_0\right\}$, then $U \le V$ in $\R
 \times [r_0,+\infty)$.
\end{pro}
\begin{proof}[Proof of Proposition \ref{prop:comp-r-grand}]
We have:
\[U_t \le\frac 1 r  G(rU_r) +\frac {U_r} r K(rU_r) + U_{rr} \sigma^2(rU_r)\]
and
\[V_t \ge\frac 1 r G(rV_r) +\frac {V_r}r K(rV_r) + V_{rr}\sigma^2(rV_r)\]
with
\[G(a)=\sqrt{1+a^2},\quad K(a)=\frac{2+a^2}{1+a^2},\quad \sigma(a)=\frac{1}{\sqrt{1+a^2}}.\]

By contradiction, assume that 
\[M=\sup_{(t,r)\in \R\times [r_0,+\infty)}\left\{U(t,r)-V(t,r)\right\}>0.\]
For $ \a,\eta>0$, we set
\[M_{\a,\eta}=\sup_{t\in \R,\ r,\rho\ge r_0}\left\{U(t,r)-V(t,\rho)
-\frac{|r-\rho|^2}{2\a}-\a\frac{r^2}2-\eta\frac{t^2}2\right\}\]
which satisfies
\begin{equation}\label{eq::ob10}
M_{\a,\eta}\ge \frac{M}{2}>0 \quad \mbox{for $\alpha,\eta$ small enough.}
\end{equation}
Since $U(t,r)-V(t,\rho)\le 2C+ \Phi_0(r)-\Phi_0(\rho)\le 2C
+L_0|r-\rho|$ (using the $L_0$-Lipschitz property of the profile
$\Phi_0$), we deduce that this supremum is reached at a point that we
denote by $(t,r,\rho)$. It satisfies
\[\eta \frac{t^2}{2} + \alpha \frac{r^2}{2} \le 2C 
- \frac{M}{2} + L_0|r-\rho| - \frac{|r-\rho|^2}{2\alpha} \le 2C
 - \frac{M}{2} + \frac{\alpha L_0^2}{2}\]
which in turn implies  (for fixed $\alpha>0$) 
\begin{equation}\label{eq:214}
\lim_{\eta \to 0} \eta t=0.
\end{equation}
\bigskip
We next distinguish two cases. 

\noindent{\bf Case 1: $r,\rho>r_0$.}  In that case, setting
$\displaystyle \tilde{U}(t,r)=U(t,r)-\alpha\frac{r^2}{2}$, we get with
$a=\tilde{U}_t(t,r),b=V_t(t,\rho), A=\tilde{U}_{rr}(t,r),
B=V_{rr}(t,\rho)$ that
\begin{equation}\label{eq:201}
a\le \frac 1 r G(rp+\a r^2)+\frac {p+\a r}r K(rp+\a r^2)+(A+\a)\sigma^2(rp+\a r^2)
\end{equation}
\begin{equation}\label{eq:202}
b\ge \frac 1 \rho G(\rho p)+\frac p \rho K(\rho p)+B\sigma^2(\rho p)
\end{equation}
$$a-b=\eta t$$
\begin{equation}\label{eq:203}
\left(\begin{array}{cc}
A&0\\
0&-B
\end{array}
\right)
\le \frac 1 \a 
\left(\begin{array}{cc}
1&-1\\
-1&1
\end{array}
\right)
\end{equation}
where $\displaystyle p:= \frac {r-\rho}\a$ satisfies (using equation
\eqref{eq:200} with $p=V_r(t,\rho)=\tilde{U}_r(t,r)$)
$$-L_0\le p,p+\a r\le -\delta.$$
Subtracting \eqref{eq:202} to \eqref{eq:201}, we get that
\begin{equation}\label{eq:210}
\eta t \le \I_1+\I_2+\I_3
\end{equation}
where
$$\I_1:=\frac 1 rG(rp+\a r^2)-\frac 1 \rho G(\rho p),\quad
\I_2:=\frac {p+\a r}r K(rp+\a r^2)-\frac p \rho K(\rho p)$$
and 
\[\I_3:=(A+\a)\sigma^2(rp+\a r^2)-B\sigma^2(\rho p).\]

\paragraph{Estimate on $\I_1$.}
We have
\begin{align*}
\I_1=&\frac 1 rG(rp+\a r^2)-\frac 1 r G(\rho p)+\frac 1 r G(\rho p)-\frac 1 \rho G(\rho p)\\
\le&G'(-r_0\delta)\left(\frac {(r-\rho)p}r+\a r\right)
+ \left(\frac {\rho-r}{r\rho}\right)G(\rho p)\\
\le &G'(-r_0\delta)\left(\frac {\a p^2}r+\a r\right)
+ \left(\frac {\rho-r}{r\rho}\right)(1+\rho|p|)\\
\le &G'(-r_0\delta)\left(\frac {\a p^2}r+\a r\right)
+\frac {\a|p|}{r\rho}+ \frac {\a p^2}r
\end{align*}
where, for the second line, we have used that $rp+\a r^2,\rho p\le
-r_0 \delta$ and $G'$ is non-decreasing on $(-\infty, -\delta r_0)$
and for the third line, we have used that $G(a)\le 1+|a|$. Choosing
$r_0$ such that $G'(-r_0\delta)\le -\frac 12$, and such that $r_0\ge
\frac 1 \delta\ge 1$ we get that
\begin{equation}\label{eq:211}
\I_1\le  -\frac 12 \a r+ {2\a}L_0^2.
\end{equation}
where we have used that $|p|\le L_0$ and $L_0\ge 1$.

\paragraph{Estimate on $\I_2$.}
Using that $K$ is bounded by $2$, we have
\[\begin{array}{ll}
\I_2 & \le  \quad \displaystyle \frac p r K(rp+\a r^2)- \frac p \rho K(\rho p) + 2 \a\\
& \le \quad \displaystyle \frac p r \left(K(rp+\a r^2)-K(\rho p)\right)+ 2\a
\end{array}\]
where we have used the fact that $p\le 0$, $\rho\ge r$.  Using now the
fact that $K$ is non-decreasing on $(-\infty, 0)$ and that $0\ge rp+\a
r^2\ge \rho p$, we get that
\begin{equation}\label{eq:212}
\I_2\le 2\a.
\end{equation}

\paragraph{Estimate on $\I_3$.}
Using the matrix inequality \eqref{eq:203}, we have that for all $\xi,
\zeta\in \R$
\[A\xi^2\le B\zeta^2 +\frac {(\xi-\zeta)^2}\a.\]
Using also that $\sigma$ is bounded by $1$, we get
\[I_3\le \a +\frac 1\a\left(\sigma (rp+\a r^2)
-\sigma (\rho p)\right)^2\le \a + \frac 1 \a 
\left(\|\sigma'\|_{L^\infty(\rho p, rp+\a r^2)}((r-\rho )p +\a r^2)\right)^2.\]
Since $|\sigma'(a)|\le \frac 1 {a^2}$, we have
$\|\sigma'\|_{L^\infty(\rho p, rp+\a r^2)}\le \frac 1 {(r(p+\alpha
  r))^2}\le \frac 1 {(r\delta)^2}$. Hence we get
\begin{equation}\label{eq:213}
I_3\le \a+ \frac 1{ \a}\left(\frac {(r-\rho )p +\a r^2} 
{(r\delta)^2} \right)^2=\a +\frac 1{\delta^4 \a} \left(\frac {\a p^2}{r^2} + \a\right)^2
\le   \a +\frac {4\a L_0^4}{\delta^4} 
\end{equation}
where for the last inequality, we have used that $r\ge r_0\ge 1$ and
$|p|\le L_0$ with $L_0\ge 1$.

Combining \eqref{eq:210}, \eqref{eq:211}, \eqref{eq:212} and
\eqref{eq:213}, we finally get
\[\eta t\le   -\frac 12\a r+ 5\a L_0^2 +\frac {4\a L_0^4}{\delta^4}.\]
Taking the limit $\eta \to 0$ and using \eqref{eq:214}, we get (using
$L_0\ge 1$)
\[0\le   -\frac 12\a r+   5\a L_0^2 +\frac {4\a L_0^4}{\delta^4}\]
which is absurd for $\displaystyle r\ge r_0>10 L_0^2+\frac{8L_0^4}{\delta ^4}$.

\paragraph{Case 2: $r=r_0$ or $\rho =r_0$.}
Assume for instance that $r=r_0$ (the case $\rho=r_0$ being similar).
Using that $M_{\a,\eta}>0$ for $\a$ and $\eta$ small enough, we get
that
\[\frac {|r_0-\rho|^2}{2\a}+\frac \a 2 r_0^2\le U(t,r_0)-V(t,\rho)\le 
V(t,r_0)-V(t,\rho)\le L_0 |r_0-\rho|.\] 
This implies in particular that $|r_0-\rho|\le 2\a L_0$. Injecting
this in the previous inequality, we obtain that
\[\frac \a2 {r_0^2}\le 2 \a L_0^2\]
which is absurd for $r_0> 2L_0$. This ends the proof of the proposition.
\end{proof}

Before proving Liouville Theorem \ref{theo:liouville}, we first prove
Theorem~\ref{theo::pp1} that has been used in Subsection \ref{s4.2}.
\begin{proof}[Proof of Theorem~\ref{theo::pp1}]
For all $\nu\in\R$, we define 
\[w^\nu=\Phi^1-\Phi^2+\nu.\]
In view of \eqref{eq:prof}, we can choose $\nu \ge 0$ big enough so
that $w^\nu\ge 0$.  We then define
\[\nu^*=\inf\{\bar \nu\ge 0 :\;   w^{\nu}\ge 0 \; \mbox{in}\; 
[0,+\infty),\quad \mbox{for all}\quad \nu\ge \bar \nu\}.\]
We want to show that $\nu^*=0$. By contradiction, let us assume that
$\nu^*>0$. Using \eqref{eq:prof}, we then have
\[
\left\{\begin{array}{l}
w^{\nu^*}\ge 0\\
w^{\nu^*}(r)>0 \quad {\rm for}\; r \textrm{ large enough}\\
\displaystyle \inf_{r\in [0,+\infty)} w^{\nu^*}(r)=0.
\end{array}\right.
\]
From Propositions \ref{prop:comp-r-petit} and \ref{prop:comp-r-grand},
we deduce that we have
\[\displaystyle \inf_{r\in \left[r_0^-,r_0^+\right]} w^{\nu^*}(r) =
\displaystyle \inf_{r\in [0,+\infty)} w^{\nu^*}(r) = 0\]
with $0<r_0^-< r_0^+$. Using again the Strong Maximum Principle
(Theorem~\ref{theo:max-fort}), we deduce that $w^{\nu^*}\equiv 0$.
For $r=+\infty$, this implies that $\nu^*=0$. Contradition. Therefore
$\nu^*=0$ and $\Phi^1\ge \Phi^2$.  Exchanging $\Phi^1$ and $\Phi^2$,
we get the reverse inequality.  This shows that $\Phi^1=\Phi^2$ and
ends the proof.
\end{proof}

We now prove Theorem~\ref{theo:liouville}.
\begin{proof}[Proof of Theorem~\ref{theo:liouville}]
The proof proceeds in several steps. 

\paragraph{Step 0: regularity and condition at $r=0$.}
Because $U$ is globally Lipschitz continuous (in space and time), we
can apply Proposition~\ref{pro::ob30} to conclude that $U\in
C^{1,2}(\R\times [0,+\infty))$. By continuity in equation
  \eqref{eq:main} up to $r=0$, we deduce that $U$ satisfies
\[U_r(t,0)=-\frac12 \quad \mbox{for all}\quad t\in\R.\]
Finally, from Lemma \ref{lem:1}, we have
\[|U_{rr}(t,r)|\le C(1+r^2) \quad \mbox{for all}\quad (t,r)\in\R\times
    [0,+\infty).\]

\paragraph{Step 1: preliminaries for the sliding method.}
We apply the sliding method (see \cite{BN}).
For any $h\in\R$, we set
\[U^h (t,r) = U(t+h,r).\]
Since $U$ satisfies \eqref{eq:dist-finie}, one can choose $b\ge 0$
large enough so that $U^h+b \ge U$ on $\R \times [0,+\infty)$. We now
consider
\[b^* = \inf \{ b\in\R : U^h+b \ge U \}\]
and we set
\[V := U^h + b^* \ge U.\]
Notice that, using in particular Step 0, we can check that the
assumptions of Propositions~\ref{prop:comp-r-petit} and
\ref{prop:comp-r-grand} are fulfilled with $0< r_0^-< r_0^+< +\infty$
(decreasing $r_0^-$ and increasing $r_0^+$ if necessary).

We claim that this implies 
\begin{equation}\label{eq::ob16}
m:= \inf_{(t,r) \in \R\times [r_0^-,r_0^+]} (V-U)=0.
\end{equation}

Indeed, if $m>0$, applying Propositions \ref{prop:comp-r-petit} and
\ref{prop:comp-r-grand}, we deduce that
\[V-U \ge m >0 \quad \mbox{on}\quad \R\times [0,+\infty)\]
which contradicts the definition of $b^*$. Therefore (\ref{eq::ob16})
holds true.

\paragraph{Step 2: consequence.} We distinguish two cases. 

\paragraph{Case 1: the infimum in (\ref{eq::ob16}) is reached at $(t_0,r_0)$.}
We have 
\[\left\{\begin{array}{ll}
V\ge U &\quad \mbox{on}\quad \R\times [0,+\infty),\\
V=U  &\quad \mbox{at}\quad (t_0,r_0) \in  \R\times [r_0^-,r_0^+].
\end{array}\right.\]
Notice that $W=V-U$ satisfies 
\[W_t = A W_r + H(rV_r)W_{rr}\]
with $A$ and $H$ defined in (\ref{eq::ob20}) and (\ref{eq::ob21}).
Moreover $A$ and $H(rV_r)$ are continuous functions because $U,V\in
C^{2,1}(\R \times [0,+\infty))$.

From the strong maximum principle (Theorem \ref{theo:max-fort})
applied to $W$, we deduce that
\[V\equiv U\]
which gives for all $k \in \Z$
\[U(t,r) = U (t+h,r)+b^* = U(t+kh,x)+kb^*.\]
In view of \eqref{eq:dist-finie}, this implies that $b^* = - \lambda
h$, i.e.
\[U (t+h,r)=U(t,r) + \lambda h.\]

\paragraph{Case 2: the infimum in (\ref{eq::ob16}) is reached at infinity.}
We now assume that there exists sequences $(t_n)_n$ and $r_n
\in[r_0^-,r_0^+]$ such that $|t_n|\to +\infty$, $r_n\to r_\infty\in
   [r_0^-,r_0^+]$ and $(V-U)(t_n,r_n)\to m$.  We define the functions
\[U_n(t,r):=U(t+t_n,r)-\lambda t_n,\quad  V_n(t,r)=V(t+t_n,r)-\lambda t_n\]
which have the same Lipschitz constant (in space and time) as the one of $U$.
We can then apply Ascoli-Arzel\`a Theorem, to deduce that, up to a subsequence, we have
\[U_n\to U_\infty,\quad V_n \to V_\infty,\quad \mbox{and}\quad V_\infty(t,r)=U_\infty(t+h,r)+b^*\]
where $U_\infty$, $V_\infty$ are two globally Lipschitz solutions of \eqref{eq:main} on $\R \times (0,+\infty)$ satisfying again
\[\left\{\begin{array}{ll}
V_\infty\ge U_\infty &\quad \mbox{on}\quad \R\times [0,+\infty),\\
V_\infty=U_\infty  &\quad \mbox{at}\quad (0,r_\infty) \in  \R\times [r_0^-,r_0^+].
\end{array}\right.\]
We can then repeat Step 0 and then case 1 for $(U,V)$ replaced by
$(U_\infty,V_\infty)$ and get that $b^* = - \lambda h$, and then $V\ge
U$ means
\begin{equation}\label{eq::ob42}
U (t+h,r)\ge U(t,r) + \lambda h.
\end{equation}

\paragraph{Step 3: conclusion.}
Notice that (\ref{eq::ob42}) means that $t\mapsto U(t,r)-\lambda t$ is
both nondecreasing (using $h>0$) and nonincreasing (using $h<0$).
This implies that
\[U(t,r)-\lambda t = U(0,r).\]
From (\ref{eq::ob43}), we have in particular
\[U_r(0,r)\le 0\]
and by our assumptions $U(0,r)$ is globally Lipschitz in the variable
$r$.  Then Theorem \ref{th:steady} i) implies that there exists a
constant $a\in\R$ such that
\[U(0,r) = \Phi(r) + a.\]
This ends the proof of the theorem.
\end{proof}

\section{Long time convergence}\label{s7}

In order to prove Theorem \ref{thm:temps-long} we need the following
proposition, whose proof is postponed.
\begin{pro}[Gradient estimate from above]\label{pro:estimate-gradient}
Let $T>0$ and let $U$ be a solution of \eqref{eq:main}-\eqref{eq:ci}
in $(0,T)\times (0,+\infty)$, such that $U$ is globally Lipschitz
continuous with respect to time. Assume that there exists a constant
$C$ such that for all $(t,r)\in (0,T)\times (0,+\infty)$, 
\begin{equation}\label{eq::ob50bis}
|U(t,r)-\lambda t-\Phi(r)|\le C.
\end{equation}
If the initial datum $U_0$ satisfies
\begin{equation}\label{eq::ob53}
(U_0)_r\le \Phi_r\quad \mbox{in} \quad (0,+\infty)
\end{equation}
then we have
\[U_r\le \Phi_r \quad \mbox{in} \quad (0,T)\times (0,+\infty).\]
\end{pro}
\begin{proof}[Proof of Theorem  \ref{thm:temps-long}]
By Theorem \ref{th::ob22}, there exists a unique solution $U$ to
(\ref{eq:main}), (\ref{eq:ci}) which is globally Lipschitz continuous
(in space and time).  Notice that $\lambda t +\Phi(r)$ is a global
solution.  Therefore, using (\ref{eq::ob26}) and applying the
comparison principe (see \cite[Theorem 1.3]{spirale1}), we deduce
the following estimate for all times,
\begin{equation}\label{eq::ob65}
|U(t,r)-\lambda t-\Phi(r)|\le C
\end{equation}
Finally using (\ref{eq::ob27}) and applying Proposition \ref{pro:estimate-gradient},
we deduce that 
\begin{equation}\label{eq:002}
U_r\le \Phi_r\le \delta <0.
\end{equation}

Then for any sequence $t_n\to +\infty$, by Ascoli-Arzel\`a theorem, we
get the convergence (for a subsequence still denoted by $(t_n)_n$),
\[U(t+t_n,r)-U(t_n,0) \to U_\infty(t,r) \quad \mbox{locally uniformly
  on compact sets}\]
 where $U_\infty$ is still globally Lipschitz
continuous and still satisfies \eqref{eq::ob65} and \eqref{eq:002}.
Therefore the Liouville result (Theorem \ref{theo:liouville}) implies
that there exists a number $a\in \R$ such that
\[U_\infty(t,r)= \lambda t + \Phi(r) +a.\]
This ends the proof of the theorem.
\end{proof}
\begin{proof}[Proof of Proposition \ref{pro:estimate-gradient}]
We have to prove that for $r> \rho>0$
\[U(t,r)-U(t,\rho)\le \Phi(r)-\Phi(\rho).\]
Using log coordinates and setting $u(t,x)=U(t,e^x)$ and $\phi(x)=\Phi(e^x)$, 
this is equivalent to prove that for $x>y>-\infty$
\[u(t,x)-u(t,y)\le \phi(x)-\phi(y).\]
Recall that $u$ and $\lambda t +\phi(x)$ are both solutions of  the following equation
\[u_t = F(x,u_x,u_{xx})=e^{-x}\sqrt{1+u_x^2}+ e^{-2x}u_x 
+ e^{-2x}\frac{u_{xx}}{1+u_x^2}.\]
By contradiction, assume that
\[M=\sup_{x>y,\  t\in [0,T)}\left\{u(t,x)-u(t,y)- \phi(x)+\phi(y)\right\}>0.\]
For $\e,\a,\eta>0$, we consider the following approximate supremum, 
\begin{equation}\label{eq::ob51bis}
M_{\e,\a,\eta}=\sup_{x>y,\ t,s\in [0,T)}\left\{u(t,x)-u(s,y)-
  \phi(x)+\phi(y)-\frac {|t-s|^2}{2\e}-\frac \a2x^2-\frac
  \eta{T-t}\right\}.
\end{equation}
Remark that when the penalization parameters $\a$ and $\eta$ are small enough, 
we have 
\[M_{\e,\a,\eta}\ge M/2 >0.\]
From (\ref{eq::ob50bis}), we deduce that $u(t,x)-u(s,y)-
\phi(x)+\phi(y)$ is bounded by $2C+\lambda T$, and then the supremum in
(\ref{eq::ob51bis}) is reached at a point that we denote by
$(t,x,s,y)$ which satisfies
\[\frac {|t-s|^2}{2\e}+\frac \a2x^2+\frac \eta{T-t}\le 2C +\lambda T.\]
We deduce in particular that 
\begin{equation}\label{eq:220}
\lim_{\a \to 0} \a x =0.
\end{equation}

The proof is divided into two cases. 

\paragraph{Case 1: there exists  $\e_n\to 0$ such that $t=0$ or $s=0$.}
Assume for example that $t=0$ (if $s=0$, a similar reasoning provides
the same contradiction). Then we have
\[\begin{array}{ll}
\displaystyle \frac \eta T & <\displaystyle u(0,x)-u(s,y)-\phi(x)+\phi(y)-\frac {s^2}{2\e}\\
& \displaystyle \le u(0,y)-u(s,y) -\frac {s^2}{2\e}\le Ls-\frac {s^2}{2\e}\le \frac {\e L^2}2
\end{array}\]
where in the second line we have used (\ref{eq::ob53}) and then used
$L$, which denotes the Lipschitz constant in time of $U$. This is
absurd for $\e$ small enough.

\paragraph{Case 2: for all $\e$ small enough we have $t,s>0$.}
In that case, using that the function
\[(t',x')\mapsto u(t',x')-u(s,y)- \phi(x')+\phi(y)
-\frac {|t'-s|^2}{2\e}-\frac \a2(x')^2-\frac \eta{T-t'}\]
 reaches a maximum at $(t,x)$, we deduce that
\[\frac {t-s}\e+\frac \eta{(T-t)^2}\le F(x,\phi_x(x)+\a x, \phi_{xx}(x)+\a).\]
Similarly, we have that 
\[\frac {t-s} \e \ge  F(y,\phi_x(y), \phi_{xx}(y)).\]
Subtracting these two inequalities, we get
\begin{align*}
\frac \eta{T^2}\le& F(x,\phi_x(x)+\a x, \phi_{xx}(x)+\a)- F(y,\phi_x(y), \phi_{xx}(y))\\
\le &F(x,\phi_x(x)+\a x, \phi_{xx}(x)+\a)-F(x,\phi_x(x),
\phi_{xx}(x))\\
& +F(x,\phi_x(x), \phi_{xx}(x))- F(y,\phi_x(y), \phi_{xx}(y))\\
\le &F(x,\phi_x(x)+\a x, \phi_{xx}(x)+\a)-F(x,\phi_x(x), \phi_{xx}(x))+\lambda -\lambda
\end{align*}
which gives
\begin{equation}\label{eq::ob66}
\frac \eta{T^2}\le e^{-x}\a |x|+ e^{-2x}\a x +e^{-2x} \a +  I
\end{equation}
with
\[I:=e^{-2x}\phi_{xx}(x)\left(\frac 1 {1+(\phi_x(x)+\a x)^2}-\frac 1 {1+(\phi_x(x))^2}\right).\]
We write
\[I:=e^{-2x}\phi_{xx}(x)\ J,\quad J:= H(\phi_x(x)+\a x)-H(\phi_x(x)),\quad H(p):=\frac{1}{1+p^2}.\]

\paragraph{Estimate on $J$.}
We observe that the function $H$ is concave in
$\left[-\frac{1}{\sqrt{3}},\frac{1}{\sqrt{3}}\right]$ and convex
outside. Recalling (\ref{eq:estGradR}), we also see that
\begin{equation}\label{eq::ob65bis}
e^{-x}\phi_x(x) = \Phi_r(e^x) \in [-1, -\lambda].
\end{equation}
We now define some $b>0$ such that 
\[\left\{\begin{array}{ll}
-\frac{1}{2\sqrt{3}}\le \phi_x(x)\le 0 & \quad \mbox{for all}\quad x\le -b<0,\\
\phi_x(x)\le -\frac{2}{\sqrt{3}}& \quad \mbox{for all}\quad x\ge b>0,
\end{array}\right.\]
We call $L_1$ the Lipschitz constant of $H$. Using (\ref{eq:220}), 
we can assume $\alpha x$ small enough.  For instance, for
$|\alpha x|\le \frac{1}{2\sqrt{3}}$, we deduce from the
convexity/concavity property of $H$ that
\begin{equation}\label{eq::ob61}
\left\{\begin{array}{ll}
\displaystyle \alpha x \frac{C}{|\phi_x(x)|^3} \ge \alpha x
H'(\phi_x(x)+\alpha x) 
\ge J\ge \alpha x H'(\phi_x(x))\ge 0 
& \quad \mbox{for all}\quad x\ge b>0,\medskip\\
\alpha x L_1\le J\le \a x H'(\phi_x(x))\le 0 & \quad \mbox{for all}\quad x\le -b<0,
\end{array}\right.
\end{equation}
where now $C>0$ is generic constant that can change from line to line.

\paragraph{Estimate on $I$.}
Notice that $\lambda t + \Phi(r)$ is a globally Lipschitz  continuous solution of (\ref{eq:main}), 
and then Lemma \ref{lem:1} implies the bound (\ref{eq::ob31}), namely
\[|\Phi_{rr}(r)|\le C (1+r^2).\]
Because $\Phi_{rr}(e^x)=e^{-2x}\phi_{xx}-e^{-2x}\phi_x$, we deduce that
\[I\le e^{-2x}\phi_x J + C(1+e^{2x}) |J|.\]
We deduce, using (\ref{eq::ob65bis}) and (\ref{eq::ob61}), that
\[I\le \left\{\begin{array}{ll}
Ce^{-x} \alpha x & \quad \mbox{for all}\quad x\ge b>0,\\
Ce^{-x} \alpha |x| + C \alpha |x| & \quad \mbox{for all}\quad x\le -b<0,
\end{array}\right.\]
which can be rewritten as
\[I \le Ce^{-x} \alpha |x| \quad \mbox{for all}\quad |x|\ge b>0.\]
Using (\ref{eq::ob66}), this leads to
\begin{equation}\label{eq::ob64}
\frac \eta{T^2}\le Ce^{-x}\a |x|+ e^{-2x}\a x   \quad \mbox{for all}\quad |x|\ge b>0
\end{equation}
We now distinguish several cases.

Assume first that there exists $\a\to 0$ such that $x\le -b$.
Increasing $b>0$ if necessary, we can assume that $Ce^{-x}\a |x|+ e^{-2x}\a
x\le 0$ for all $x\le -b$, which gives a contradiction.

Second, assume that there exists $\a\to 0$ such that $x\ge b$.
For $x\ge b>0$, sending $\alpha\to 0$ in (\ref{eq::ob64}), we get a contradiction.

Finally, assume that for all $\a$ small enough, we have $-b \le x \le b$.
In that case, we have from (\ref{eq::ob66})
\[\frac \eta{T^2}\le e^{-x}\a |x|+ e^{-2x}\a x +e^{-2x} \a +e^{-2x}|\phi_{xx}(x)|L_1\a |x|.\]
Again, sending $\a\to 0$, we get a contradiction.  This ends the proof
of the proposition.
\end{proof}

\small

\appendix

\section{Appendix}\label{s8}

\subsection{Strong maximum principle and Hopf lemma}

In this subsection, we recall the classical strong maximum principle
and Hopf lemma.\\ For $-\infty\le t_1<t_2\le +\infty$ and $0<R\le
+\infty$, let us consider the following general linear parabolic
equation
\begin{equation}\label{quasi-lin}
w_t = a(t,r) w_{rr} + b(t,r) w_r + c(t,r) w \quad \mbox{on}\quad Q:=(t_1,t_2)\times (0,R)
\end{equation}
with the following assumptions on the coefficients
\begin{equation}\label{eq::ob15}
\left\{\begin{array}{l}
a,b,c\in C(\overline{Q}),\\
a\ge \delta >0 \quad \mbox{on}\quad \overline{Q}
\end{array}\right.
\end{equation}
For $A=Q$ or $\overline{Q}$, we recall that we say that a function
$w\in C^{2,1}(A)$ if and only if $w,w_r, w_{rr}, w_t\in C(A)$.  Then
we have the following classical result.

\begin{theo}[Strong maximum principle \cite{nirenberg53}]\label{theo:max-fort}
Consider a function $w \in C^{2,1}(Q)$ which is a supersolution of
\eqref{quasi-lin}.  If
\[\left\{\begin{array}{ll}
w\ge 0 & \quad \mbox{on}\quad Q,\\
w=0 & \quad \mbox{at}\quad (t_0,r_0)\in Q
\end{array}\right.\]
then $w\equiv 0$ on $Q\cap \left\{t\le t_0\right\}$.
\end{theo}
We also have (see \cite[Lemma~2.8]{lieberman96}).
\begin{lem}[Hopf lemma]\label{lem:hopf}
Consider a function $w\in C^{2,1}(\overline{Q})$ which is a
supersolution of \eqref{quasi-lin}.  If
\[\left\{\begin{array}{ll}
w\ge 0 & \quad \mbox{on}\quad \overline{Q},\\
w=0 & \quad \mbox{at}\quad (t_0,0)\in \partial Q \quad \mbox{with}\quad t_0\in (t_1,t_2)
\end{array}\right.\]
then either $w \equiv 0$ on $\overline{Q}\cap \left\{t\le t_0\right\}$
or $w_r(t_0,0)>0$.
\end{lem}

\subsection{Interior Schauder estimate}

The following result can be found in Krylov \cite{K} (see also \cite{LSU,lieberman96}).
\begin{pro}[Interior Schauder estimates]\label{pro::ob36}
Let $T>0$, $\delta>0$, $R>0$ and $\alpha\in (0,1)$ and $N\ge 1$. 
Assume that $w$ solves (in the sense of distributions)
\[w_t = a \Delta w + b \quad \mbox{on}\quad (T-\delta,T+\delta)\times B_R\] 
with $B_R$ the ball of radius $R$ in $\R^N$. Assume that $a,b
\in C_{t,x}^{\frac{\alpha}{2},\alpha}((T-\delta,T+\delta)\times B_R)$
with for some $\eta>0$:
\[0< \eta \le a \le 1/\eta \quad \mbox{on}\quad (T-\delta,T+\delta)\times B_R\]
and
\[\|a\|_{C_{t,x}^{\frac{\alpha}{2},\alpha}((T-\delta,T+\delta)\times B_R)}\le C_0.\]
Then there exists a constant $C=(\delta,R,\alpha,N,\eta,C_0)>0$ such that 
\[\|w\|_{C_{t,x}^{\frac{\alpha}{2},\alpha}([T,T+\delta)\times
    B_{R/2})} \le
  C\left\{\|b\|_{C_{t,x}^{\frac{\alpha}{2},\alpha}((T-\delta,T+\delta)\times
    B_R)} + |w|_{L^\infty((T-\delta,T+\delta)\times B_R)}\right\}.\]
\end{pro}

\subsection{A technical lemma}

\begin{lem}[A H\"{o}lder estimate]\label{lem::ob35}
Let $\alpha\in (0,1)$ and $N\ge 1$.  For $X\in\R^N$, let us define the function
\[\zeta({X}):=\left\{\begin{array}{ll}
\displaystyle |X|^\alpha \frac{X}{|X|} & \quad \mbox{if}\quad X\not=0,\\
0  & \quad \mbox{if}\quad X=0
\end{array}\right.\]
Then there exists a constant $C=C(\alpha)>0$ such that for all $X',X\in \R^N$, we have
\[|\zeta({X}') - \zeta({X})|\le C |X'-X|^\alpha.\]
\end{lem}
\begin{proof}[Proof of Lemma \ref{lem::ob35}]
Let us assume that $|X'|\ge |X|>0$. We write
\[\zeta(X')-\zeta(X) = T_1 + T_2\]
with
\[T_1=\left(|X'|^\alpha-|X|^\alpha\right)\frac{X'}{|X'|}\quad
\mbox{and}\quad
T_2=|X|^\alpha\left(\frac{X'}{|X'|}-\frac{X}{|X|}\right).\]

\paragraph{Step 1: estimate on $T_1$.}
We have
\begin{equation}\label{eq::ob45}
||X'|^\alpha-|X|^\alpha| = |X|^\alpha\left|r^\alpha -1\right| \quad
\mbox{with}\quad r = \frac{|X'|}{|X|}\ge 1.
\end{equation}

\paragraph{Case A: $1\le r\le 2$.}
We write $r=1+\delta$ with $0\le \delta\le 1$. Then we have
\[\begin{array}{ll}
\left|r^\alpha -1\right| & = \alpha \delta + O(\delta^2)\\
& \le C \delta^\alpha = C |r-1|^\alpha.
\end{array}\]
\paragraph{Case B: $r\ge 2$.}
Then we have
\begin{equation}\label{eq::ob44}
\left|r^\alpha -1\right| \le C |r-1|^\alpha.
\end{equation}
Putting together cases A and B, we see that (\ref{eq::ob44}) holds
true for any $r\ge 1$.  Using (\ref{eq::ob45}), we get for some $C\ge
1$:
\begin{equation}\label{eq::ob46}
|T_1|=||X'|^\alpha-|X|^\alpha| \le C ||X'|-|X||^\alpha \le C |X'-X|^\alpha.
\end{equation}

\paragraph{Step 2: estimate on $T_2$.}
Writing $\displaystyle e=\frac{X}{|X|}$, $\displaystyle
Y=\frac{X'}{|X|}$ with $|Y|\ge 1$, and using the fact that the map
$\displaystyle Z\mapsto \frac{Z}{|Z|}$ is $1$-Lipschitz (for the
euclidean norm) on $\R^N\backslash B(0,1)$, we get that
\begin{equation}\label{eq::ob50}
\left|\frac{X'}{|X'|}-\frac{X}{|X|}\right|=\left|\frac{Y}{|Y|}-\frac{e}{|e|}\right|\le
|Y-e| = \frac{\left|X'-X\right|}{|X|}.
\end{equation}

\paragraph{Case A: $|X|\le |X'|\le 2|X|$.}
Using (\ref{eq::ob50}), we deduce that
\[|T_2|\le |X|^\alpha \frac{|X'-X|}{|X|} = 
\frac{|X'-X|^{1-\alpha}}{|X|^{1-\alpha}}|X'-X|^\alpha \le
2^{1-\alpha}\left(\frac{|X'-X|}{|X'|}\right)^{1-\alpha}|X'-X|^\alpha\]
which implies
\begin{equation}\label{eq::ob51}
|T_2|\le  4^{1-\alpha}|X'-X|^\alpha.
\end{equation}

\paragraph{Case B: $|X'|\ge 2|X|$.}
We have
\[|T_2|\le |X|^\alpha \le ||X'|-|X||^\alpha \le |X'-X|^\alpha\]
Putting together cases A and B, we see that (\ref{eq::ob51}) holds
true for any $|X'|\ge |X|>0$.

\paragraph{Step 3: conclusion.}
From Steps 1 and 2, we deduce that there exists a constant $C>0$ such
that
\[|\zeta(X')-\zeta(X)|\le C |X'-X|^\alpha\]
This last estimate is also true if $X=0$. By symmetry between $X'$ and
$X$, we see that it is finally true for any $X,X'\in\R^N$.  This ends
the proof of the lemma.
\end{proof}

\subsection{Equation satisfied by the curvature}

The following result is not used in the rest of the paper.
We give it as an interesting result of independent interest.
\begin{lem}[Equation satisfied by the curvature]\label{lem::*16}
Let $\Phi$ be the profile given by Theorem~\ref{th:steady}.
The curvature $\kappa(x)=\kappa_\Phi(e^x)$ solves the following equation
\begin{equation}\label{eq::*17}
\kappa_t=\frac{e^{-2x} \kappa_{xx}}{1+u_x^2} + \kappa^2(1+\kappa) 
+ e^{-2x} \kappa_x\left\{-1 + \frac{2u_x^2}{1+u_x^2} +\frac{e^x u_x}{\sqrt{1+u_x^2}} \right\}.
\end{equation}
\end{lem}
\begin{proof}[Proof of Lemma \ref{lem::*16}]
We start from
\[u_t= e^{-x}\sqrt{1+u_x^2}(1+\kappa)\]
with
\[\kappa = e^{-x}\left(\frac{u_x}{\sqrt{1+u_x^2}}+\frac{u_{xx}}{(1+u_x^2)^{\frac32}}\right).\]
Let us define
\[M(a)=\frac{a}{\sqrt{1+a^2}}\quad \mbox{with}\quad M'(a)=\frac{1}{(1+a^2)^{\frac32}}.\]
Then we can write
\[\kappa = e^{-2x}(e^x M(u_x))_x\]
and
\[\kappa_t=e^{-2x}(e^x M'(u_x)u_{xt})_x.\]
We now compute
\[\begin{array}{ll}
u_{xt} & \displaystyle =e^{-x}\sqrt{1+u_x^2}(\kappa_x -(1+\kappa)) +
\frac{e^{-x}}{\sqrt{1+u_x^2}}(1+\kappa)u_xu_{xx}\medskip \\
&\displaystyle =e^{-x}\sqrt{1+u_x^2}(\kappa_x -(1+\kappa)) + e^{-x}\sqrt{1+u_x^2}(1+\kappa)u_x \left(\frac{u_{xx}}{1+u_x^2}+u_x-u_x\right)\medskip\\
&\displaystyle =e^{-x}\sqrt{1+u_x^2}\kappa_x - e^{-x}\sqrt{1+u_x^2}(1+\kappa)(1+u_x^2) +
(1+u_x^2)(1+\kappa)\kappa u_x  \medskip \\
&\displaystyle =(1+\kappa) \left\{(1+u_x^2)\kappa u_x   - e^{-x}(1+u_x^2)^{\frac32} 
\right\}   + e^{-x}\sqrt{1+u_x^2}\kappa_x.
\end{array}\]
This gives
\[\begin{array}{ll}
e^{2x}\kappa_t &\displaystyle =\partial_x\left( (1+\kappa) \left\{\kappa \frac{e^x u_x}{\sqrt{1+u_x^2}}   - 1 \right\}   
+ \frac{\kappa_x}{1+u_x^2} \right)\medskip \\
&\displaystyle =\partial_x\left( (1+\kappa) \left\{\kappa e^x M(u_x)
- 1 \right\} + \frac{\kappa_x}{1+u_x^2} \right)\medskip \\
&\displaystyle  =\kappa_x\left\{\kappa e^x M(u_x)   
- 1 +(1+\kappa)e^x M(u_x) -\frac{2u_xu_{xx}}{(1+u_x^2)^2}\right\} 
+ \frac{\kappa_{xx}}{1+u_x^2} +(1+\kappa)\kappa e^{2x}\kappa.
\end{array}\]
Therefore
\[\begin{array}{ll}
\kappa_t & \displaystyle =\frac{e^{-2x}\kappa_{xx}}{1+u_x^2} +(1+\kappa)\kappa^2 \\
 & + e^{-2x}\kappa_x \left\{-1+(1+2\kappa)e^x M(u_x) -2e^x M(u_x)(\kappa
-e^{-x}M(u_x))\right\}\medskip \\
&\displaystyle =\frac{e^{-2x}\kappa_{xx}}{1+u_x^2} +(1+\kappa)\kappa^2 
+ e^{-2x}\kappa_x \left\{-1+e^x M(u_x) +2 (M(u_x))^2\right\}\\
\end{array}\]
which shows the result. This ends the proof of the lemma.
\end{proof}

\paragraph{Acknowledgements.} The authors would like to thank the
referee for the careful reading of the paper, his/her corrections of
some mistakes and for some indications on the literature.  This work
is partially supported by the ANR projects HJnet ANR-12-BS01-0008-01,
AMAM ANR 10-JCJC 0106 and IDEE ANR-2010-0112-01.  R.M. thanks
R.V. Kohn for indications on the literature.  N.F. and R.M. thank the
conference center of Oberwolfach for providing them excellent research
conditions during the preparation of this work.  They also thank the
organizers of the meeting ``Interfaces and Free Boundaries: Analysis,
Control and Simulation'' in 2013, for the invitation to participate
and present their works.  R.M. also thank Y. Giga for an invitation to
a meeting in Sapporo in 2010 and stimulating and enlighting
discussions about the problem studied in this paper.


\end{document}